\documentclass[sort&compress,3p]{elsarticle}

\usepackage{algorithm}
\usepackage{algpseudocode}
\usepackage{amsmath}
\usepackage{amssymb}
\usepackage{amsthm,dsfont}
\usepackage{array}
\usepackage{booktabs}
\usepackage{breqn}
\setkeys{breqn}{breakdepth={1}}
\usepackage{changepage}
\usepackage{collcell}
\usepackage{color}
\usepackage{diagbox}
\usepackage{dsfont}
\usepackage{enumerate}
\usepackage{enumitem}
\usepackage{epstopdf}
\usepackage{float}
\usepackage[T1]{fontenc}
\usepackage{geometry}
\usepackage{graphics}
\usepackage{graphicx}
\usepackage{hyperref}
\usepackage[labelfont=bf,labelsep=space]{caption}
\usepackage{latexsym}
\usepackage{longtable}
\usepackage{lscape}
\usepackage{mathtools}
\usepackage{mathrsfs}
\usepackage{mdframed}
\usepackage{pdfpages}
\usepackage{pgf}
\usepackage{pgfplots}
\usepackage{pifont}
\usepackage{placeins}
\usepackage{siunitx}
\usepackage{soul}
\usepackage{subcaption}
\usepackage{tabularx}
\usepackage{tikz}
\usetikzlibrary{shapes}
\usepackage{todonotes}
\usepackage{wasysym}
\usepackage{ulem}
\usepackage[utf8]{inputenc}
\usepackage{xcolor}
\usepackage{xspace}
\usepackage{cleveref}
\pgfplotsset{compat=1.14}
\usetikzlibrary{arrows}

\usepackage{newfloat}
\usepackage{caption}
\DeclareFloatingEnvironment[fileext=frm,placement={!htb},name=Algorithm]{formulation}
\DeclareFloatingEnvironment[fileext=mcoeff,placement={!htb},name=Method]{method}
\newtheorem{theorem}{Theorem}[section]

\theoremstyle{remark}
\newtheorem*{remark}{Remark}

\theoremstyle{definition}
\newtheorem{definition}{Definition}[section]

\theoremstyle{assumption}
\newtheorem{assumption}{Assumption}

\theoremstyle{example}
\newtheorem{example}{Example}
\DeclareMathAlphabet{\mathsf}{OT1}{\sfdefault}{m}{n}
\SetMathAlphabet{\mathsf}{bold}{OT1}{\sfdefault}{b}{n}

\DeclarePairedDelimiter\abs{\lvert}{\rvert}%
\DeclarePairedDelimiter\norm{\lVert}{\rVert}%

\makeatletter
\let\oldabs\abs
\def\abs{\@ifstar{\oldabs}{\oldabs*}}
\let\oldnorm\norm
\def\norm{\@ifstar{\oldnorm}{\oldnorm*}}
\makeatother

\newcommand{\f}{\mathsf{f}}
\newcommand{\fone}{\f^{\{1\}}}
\newcommand{\fp}{\f^{\{p\}}}
\newcommand{\ftwo}{\f^{\{2\}}}

\newcommand{\g}{\mathsf{g}}

\newcommand{\gone}{\g^{\{1\}}}
\newcommand{\gp}{\g^{\{p\}}}

\newcommand{\gtwo}{\mathsf{\g}^{\{2\}}}
\newcommand{\Zero}{\mathsf{0}}
\newcommand{\Id}{\mathsf{I}}
\newcommand{\J}{\mathsf{J}}

\newcommand{\Jone}{\J^{\{1\}}}

\newcommand{\Jtwo}{\mathsf{\J}^{\{2\}}}
\newcommand{\landauO}{\mathcal{O}}
\newcommand{\Lb}{\mathsf{L}}
\newcommand{\Lone}{\mathsf{L}^{\!\{1\}}}
\newcommand{\Lp}{\mathsf{L}^{\{p\}}}
\newcommand{\Ltwo}{\mathsf{L}^{\!\{2\}}}

\newcommand{\one}{{\{1\}}}
\newcommand{\pfrac}[2]{\frac{\partial #1}{\partial #2}}

\newcommand{\tn}{t_n}
\newcommand{\two}{{\{2\}}}

\newcolumntype{P}[1]{>{\centering\arraybackslash\vspace{-1ex}}m{#1}<{\vspace*{-1ex}}}

\begin{document}

\begin{frontmatter}

\title{Efficient implementation of partitioned stiff exponential Runge-Kutta methods}
\author[aff1]{Mahesh Narayanamurthi}
\ead{maheshnm@vt.edu}
\author[aff1]{Adrian Sandu}
\ead{sandu@cs.vt.edu}
\address[aff1]{Computational Science Laboratory, Department of Computer Science, Virginia Tech, Blacksburg, VA 24060}

\includepdf[landscape=false,pages=-]{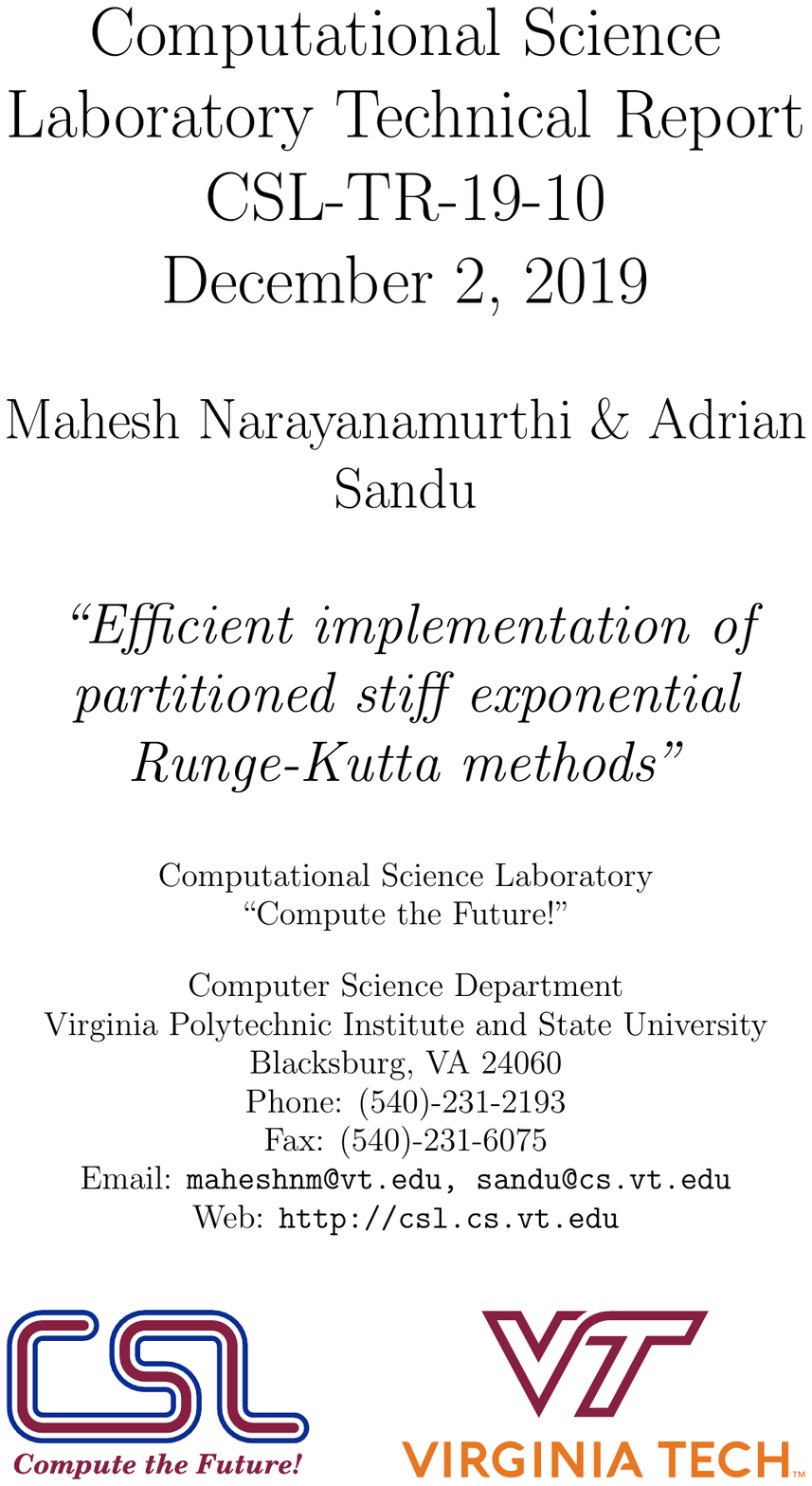}

\begin{abstract}

Multiphysics systems are driven by multiple processes acting simultaneously, and their simulation leads to partitioned systems of differential equations. This paper studies the solution of partitioned systems of differential equations using exponential Runge-Kutta methods. We propose specific multiphysics implementations of exponential Runge-Kutta methods satisfying stiff order conditions that were developed in [Hochbruck et al.,  SISC, 1998] and [Luan and Osterman, JCAM, 2014]. We reformulate stiffly--accurate exponential Runge--Kutta methods in a way that naturally allows of the structure of multiphysics systems, and discuss their application to both component and additively partitioned systems. The resulting partitioned exponential methods only compute matrix functions of the Jacobians of individual components, rather than the Jacobian of the full, coupled system. We derive modified formulations of particular methods of order two, three and four, and apply them to solve a partitioned reaction-diffusion problem. The proposed methods retain full order for several partitionings of the discretized problem, including by components and by physical processes.   
\end{abstract}


\end{frontmatter}

\newpage
\tableofcontents

\newpage
\section{Introduction}
\label{sec:introduction}
%

Implicit methods are commonly used to solve stiff systems of ODEs since explicit methods, although cheaper per timestep, suffer significantly from stability restrictions. For large systems implicit methods become expensive as they involve the solution of (non-)linear systems of algebraic equations at each timestep. Widely used Krylov-based iterative solvers can be slow to converge in the absence of a good preconditioner \cite{rainwater2016}. For non-linear problems, obtaining a preconditioner can be challenging.

Multi-methods for time integration, such as implicit-explicit (IMEX) \cite{Ascher_1995_IMEX_LMM,Ascher_1997_IMEX_RK,Verwer_2004_IMEX_RKC,Sandu_2010_extrapolatedIMEX,Sandu_2012_ICCS-IMEX,Sandu_2014_IMEX_GLM_Extrap,Sandu_2014_IMEX-GLM,Sandu_2014_IMEX-RK,Sandu_2015_IMEX-TSRK,Sandu_2015_Stable_IMEX-GLM,Sandu_2016_highOrderIMEX-GLM} address the cost challenge by treating implicitly only the stiff components of the system, and the rest of the system explicitly. In a multi-physics system, with two or more physics based partitions, the implicit method is applied to the stiff partition and the explicit method to the non-stiff partition. IMEX methods can be inefficient when applied to systems where both partitions are stiff. Implicit-Implicit methods \cite{belytschko1979,Zienkiewicz1988,Farhat1991,Dettmer2012,Sandu_2015_GARK,Sandu_2016_GARK-MR} treat implicitly each component, but are less stable compared to a monolithic implicit integrator \cite{belytschko1979}.

Exponential time integrators solve the linear part of a problem exactly \cite{Hochbruck_2010_exp,minchev2005}, are more stable than explicit methods, and can be cheaper than implicit methods per time step \cite{loffeld2013,Sandu_2018_EPIRK-adjoint}. Implicit-Exponential methods \cite{luan2016,nie2006,chou2007,zhao2011} that treat one partition implicitly and the other partition exponentially have been studied in the recent years  \cite{luan2016,nie2006, chou2007}.

Our recent work \cite{narayanamurthi2019} focuses on partitioned methods that solve each component of the system with a different exponential scheme. In \cite{narayanamurthi2019} we developed the classical order condition theory for partitioned exponential methods constructed in the EPIRK framework \cite{tokman2006,tokman2011} using a GARK structure\cite{Sandu_2015_GARK,Sandu_2016_GARK-MR}, as well as constructed in the Rosenbrock-exponential framework \cite{hochbruck1997,Hochbruck_1998_exp}.  We found that partitioned methods satisfying the classical order conditions are amenable to certain computational optimizations that are unavailable to unpartitioned methods, can be more stable than unpartitioned methods for some experimental settings. However, as the stiffness of the coupled system increases, unpartitioned methods remain superior to partitioned methods. 

In this work we study the implementation of exponential Runge--Kutta methods that satisfy stiff order conditions \cite{hochbruck2009,Hochbruck_2005_expRK,luan2014,luan2014a,luan2014b} on partitioned systems of differential equations. Special choices of the linear operator lead to partitioned exponential schemes that only compute matrix functions of the Jacobians of individual components, rather than the Jacobian of the full, coupled system. We reformulate stiffly--accurate exponential Runge--Kutta methods \cite{Hochbruck_2005_expRK} in a way that naturally allows of the structure of multiphysics systems, and discuss their application to both component and additively partitioned systems.

The remainder of the paper is organized as follows.
Section \ref{sec:partitioned-systems} discusses component and additively partitioned systems of differential equations. Section \ref{sec:exprk_methods_w_conditions} reviews the exponential Runge--Kutta methods and stiff order conditions developed in \cite{hochbruck2009,Hochbruck_2005_expRK,luan2014,luan2014a,luan2014b}. Implementation of split exponential Runge-Kutta schemes for component partitioned systems is discussed in Section \ref{sec:pexprk-on-component-partitioning}, and for additively partitioned systems in Section \ref{sec:partitioned_exprk_methods}. Coefficients for particular methods satisfying the stiff order conditions are given, in the transformed and partitioned formulations, in Section \ref{sec:particular_methods}. Considerations related to computational cost of partitioned exponential integrators are given in Section \ref{sec:computational-optimizations}. Numerical experiments and results are discussed in Section \ref{sec:results}, and conclusions are drawn in Section \ref{sec:conclusions}.


\section{Partitioned systems of differential equations}
\label{sec:partitioned-systems}

In this paper we are concerned with solving partitioned initial value problems arising in multiphysics simulations. Additively partitioned systems 
\begin{equation}
\label{eqn:additively-split-ode}
u' = \f(u) = \fone(u) + \ftwo(u), \qquad u(\tn) = u_n \in \mathbb{R}^N,
\end{equation}
are characterized by a full right-hand side $\f$ that is the sum of two components $\fone$ and $\ftwo$,  with Jacobians 
\begin{equation}
\label{eqn:jacobians}
\begin{split}
\f' = \Jone+\Jtwo,\qquad
\Jone \coloneqq \pfrac{\fone}{u},\qquad
\Jtwo  \coloneqq \pfrac{\ftwo}{u}.
\end{split}
\end{equation}
Component partitioned systems split the state vector into components $v$ and $w$, and have the form
\begin{equation}
\label{eqn:component-partitioned-ode}
y' = \begin{bmatrix} v' \\ w' \end{bmatrix} = \begin{bmatrix} \fone(v, w) \\ \ftwo(v, w) \end{bmatrix} = \f(y), \qquad
\f'(y)= \begin{bmatrix} \Jone & \pfrac{\fone}{w} \\ \pfrac{\ftwo}{v} &\Jtwo \end{bmatrix},\qquad
\Jone \coloneqq \pfrac{\fone}{v},\qquad
\Jtwo  \coloneqq \pfrac{\ftwo}{w}.
\end{equation}
Using the variable transformation $u = v + w$ \cite{Ascher_1995_IMEX_LMM, Ascher_1997_IMEX_RK} the additively partitioned system \eqref{eqn:additively-split-ode} can be written in the component partitioned form \eqref{eqn:component-partitioned-ode}:
\begin{equation}
\label{eqn:component-form-additive-ode}
\begin{split}
y' = \begin{bmatrix} v' \\ w' \end{bmatrix}
=  \begin{bmatrix} \fone(v+ w) \\ \ftwo(v+ w) \end{bmatrix} =
\f(y), \qquad \J \coloneqq \f'(y) = \begin{bmatrix} \Jone & \Jone \\ \Jtwo & \Jtwo \end{bmatrix}.
\end{split}
\end{equation}
We note that  by adding the two coupled equations \eqref{eqn:component-form-additive-ode} we recover the original ODE \eqref{eqn:additively-split-ode}.

\section{Exponential Runge-Kutta methods and W order conditions}
\label{sec:exprk_methods_w_conditions}

Application of exponential methods to solve a general ODE system, 
\begin{equation}
\label{eqn:general-ode}
y' = \f(y), \qquad y(t_n) = y_n \in \mathbb{R}^N,
\end{equation}
is based on splitting the right hand side into a linear term with a fixed matrix $\Lb$ and a nonlinear reminder $\g$:
\begin{equation}
\label{eqn:lin-nonlin-split-ode}
y' = \f(y) = \Lb\cdot y + \g(y), \qquad \Lb \approx \f'(y), \qquad \g(y) \coloneqq \f(y) - \Lb\cdot y,
\end{equation}
where $\Lb$ captures the stiffness of the right hand side and therefore needs to approximate well the stiff directions of the Jacobian $\f'(y)$.

An exponential Runge--Kutta method \cite{Hochbruck_2005_expRK,hochbruck2009,luan2014}  applied to the semi-linearly split system \eqref{eqn:lin-nonlin-split-ode} reads:
\begin{equation}
\label{eqn:exp-RK}
\begin{split}
Y_i &= y_n + h \, c_i\, \varphi_{1}(c_i h \Lb)\, \f(y_n) + h \, \sum_{j = 2}^{i - 1} a_{i,j}(h \Lb)\, \Big( \g(Y_j) - \g(y_n) \Big), \qquad i = 1, \hdots s, \\
y_{n + 1} &= y_n + h \,  \varphi_{1}(h \Lb)\, \f(y_n) + h \, \sum_{j = 2}^{s} b_j(h \Lb) \, \Big( \g(Y_j) - \g(y_n) \Big).
\end{split}
\end{equation} 
Here, $y_n$ is the current solution, $Y_i$ are the internal stages of the method, $s$ is the number of internal stages in the method, $y_{n+1}$ is the numerical solution after one step, $\varphi_1$ is the analytic function with index $k = 1$ in the sequence 
\begin{equation}
	\label{eqn:phi_function_definition}
	\begin{split}
&	\varphi_0(z) = \exp(z),\qquad
		\varphi_{k}(z) = \int_{0}^{1} e^{(1 - \theta) z} \cdot \frac{\theta^{k - 1}}{(k-1)!} \,d\theta = \sum_{i=0}^{\infty} \cfrac{z^i}{(k + i)!},\qquad k \ge 1, \\
&	\varphi_{k+1}(z) = \cfrac{\varphi_k(z) - \varphi_{k}(0)}{z},\qquad 
	\varphi_{k}(0) = \frac{1}{k!}.
	\end{split}
\end{equation}
The coefficients $a_{i,j}$ and $b_j$ are analytic functions, typically linear combinations of $\varphi_k$ (and scaled versions), and $c_i$ are the abscissae corresponding to the internal stages. These coefficients can be expressed in a Butcher tableau form \cite[Chapter 2]{Butcher_2016_book} as follows:
\[
\renewcommand\arraystretch{1.2}
\begin{array}
{c|c}
\mathbf{c} & \mathbf{A} \\
\hline
& \mathbf{b}^{T}
\end{array}
\]
where the $a_{i,j}$ coefficients are arranged in the form of a matrix $\mathbf{A}$, the $c_i$ coefficients in vector $\mathbf{c}$ and $b_j$ coefficients in row vector $\mathbf{b}^{T}$.

Stiff order conditions for exponential Runge--Kutta methods are given in  \cite[Table 1]{luan2014b}, and are repeated here in Table \ref{tbl:expRK-order-conditions} (up to order four) for convenience. These order conditions hold for any choice of the matrix $\Lb$. The functions $\psi_{i,j}(h\Lb)$ are defined as:
\begin{equation}
\label{eqn:psi-ij}
    \psi_{i,j}(h\Lb) \coloneqq \sum_{k=2}^{j-1} a_{j,k}(h\Lb)\, \frac{c_k^{j-1}}{(j-1)!} - c_j^i\,\varphi_i(c_j h\Lb).  
\end{equation}
There is one minor notational difference. In the original paper \cite{luan2014b}, conditions 3b, 4b, and 4c of Table \ref{tbl:expRK-order-conditions} are written without any assumptions made about the Jacobian of $\g$, and therefore $\g'$ is an arbitrary square matrix. Here we replace $\g'$ with $\J - \Lb$ following equations \eqref{eqn:block-diagonal-L-component}, \eqref{eqn:block-diagonal-L-additive}. 

In \cite{luan2014,luan2014a,luan2014b,hochbruck2009} the following assumptions are made.

\begin{assumption}
\label{ass:one}
There exist constants $M$ and $\omega$ such that
\begin{equation}
\label{eqn:semigroup}
\Vert e^{t\,\Lb} \Vert \le M\, e^{t\,\omega}, \quad t\ge 0.
\end{equation}
As a consequence, the functions $\varphi_k(h\Lb)$ \eqref{eqn:phi_function_definition} as well as the coefficients $a_{i,j}(h\Lb)$ and $b_j(h\Lb)$ are bounded operators.
\end{assumption}

\begin{assumption}
\label{ass:two}
The equation \eqref{eqn:lin-nonlin-split-ode} admits a sufficiently smooth solution $y : [0,T] \to \mathbb{R}^N$. The nonlinear function $\g$ is sufficiently many times (Fr\'{e}chet) differentiable in a neighborhood of the solution, with uniformly bounded derivatives.
\end{assumption}
It is shown in \cite{luan2014,luan2014a,luan2014b,hochbruck2009} that, if Assumptions \ref{ass:one} and \ref{ass:two} are satisfied, then a method \eqref{eqn:exp-RK} that satisfies the order conditions of Table \ref{tbl:expRK-order-conditions} up to order $p$ computes a numerical solution $y_n$ that converges to the exact solution with order $p$.
\begin{table}[tb!]
    \centering
    \begin{tabular}{|c|c|c|}
        \hline 
        No. & Order & Order Condition \\ \hline
        1. &  1 & \(\displaystyle \sum_{j=1}^s b_j(h\Lb) = \varphi_1(h\Lb) \), \\ \hline
        2a. & 2 & \(\displaystyle \sum_{j=2}^s b_j(h\Lb)\, c_j = \varphi_2(h\Lb) \),  \\ \hline
        2b. & 2 & \(\displaystyle \sum_{j=1}^s a_{i,j}(h\Lb) = c_i \varphi_1(h\Lb), \quad i = 2, \dots, s \),\\ \hline
        3a. & 3 & \(\displaystyle \sum_{j=2}^s b_j(h\Lb) \frac{c_j^2}{2} = \varphi_3(h\Lb) \), \\ \hline 
        3b. & 3 & \(\displaystyle \sum_{j=2}^s b_j(h\Lb) \, \left(\J-\Lb \right) \, \psi_{2,j}(h\Lb)= 0 \), \\ \hline
        4a. & 4 & \(\displaystyle \sum_{j=2}^{s} b_j(h\Lb)  \frac{c_j^3}{3!} = \varphi_4(h\Lb) \), \\ \hline
        4b. & 4 & \(\displaystyle \sum_{j=2}^s b_j(h\Lb) \, \left(\J-\Lb \right) \, \psi_{3,j}(h\Lb)= 0 \). \\ \hline
        4c. & 4 & \(\displaystyle \sum_{j=2}^s b_j(h\Lb) \, \left(\J-\Lb \right) \, \sum_{k=2}^{j-1} a_{jk}(h\Lb) \, \left(\J-\Lb \right) \, \psi_{2,k}(h\Lb)= 0 \). \\ \hline
        4d. & 4 & \(\displaystyle \sum_{j=2}^s b_j(h\Lb) \,  c_j \, \mathsf{K} \, \psi_{2,j}(h\Lb)= 0 \). \\ \hline
    \end{tabular}
    \caption{Stiff oder conditions for exponential Runge--Kutta methods (up to order four) \cite{luan2014b}, where $\Lb$ and $\mathsf{K}$ are arbitrary matrices.}
    \label{tbl:expRK-order-conditions}
\end{table}
%

\section{Split exponential Runge-Kutta schemes for component partitioned systems}
\label{sec:pexprk-on-component-partitioning}
%
Application of exponential methods to the component partitioned system \eqref{eqn:component-partitioned-ode} requires a linear-nonlinear split of the form \eqref{eqn:lin-nonlin-split-ode} where $\Lb \approx \f'$ along the stiff subspace. The matrix $\Lb$ has the same dimension as $\f'$, and in general has no particular structure. Evaluation of matrix functions $\varphi_k(h\Lb)$ is computationally expensive.

In this paper we seek to leverage the partitioned structure of the system \eqref{eqn:component-partitioned-ode} in order to build computationally efficient implementations of exponential Runge-Kutta methods.

Specifically, we consider a linear-nonlinear splitting \eqref{eqn:lin-nonlin-split-ode} of the component partitioned system \eqref{eqn:component-partitioned-ode} of the form
\begin{equation}
\label{eqn:component-split-ode-semilinear}
\begin{split}
y' &= \begin{bmatrix} v' \\ w' \end{bmatrix} =
\begin{bmatrix} \fone(v, w) \\ \ftwo(v, w) \end{bmatrix} =
\begin{bmatrix}  \Lone \cdot v  + \gone(v, w) \\ \Ltwo \cdot w + \gtwo(v, w) \end{bmatrix} 
= \Lb \cdot y  + \g(y).
\end{split}
\end{equation}
where $\Lone$ and $\Ltwo$ are linear operators that contain the stiffnesses of $\fone$ and $\ftwo$, respectively, $\gone$ and $\gtwo$ are the corresponding non-linear remainder functions. This splitting is of the form \eqref{eqn:lin-nonlin-split-ode} with
\begin{equation}
\label{eqn:block-diagonal-L-additive}
\Lb \coloneqq \begin{bmatrix} \Lone & \Zero \\  \Zero & \Ltwo \end{bmatrix}, \quad
\g(v,w) \coloneqq \begin{bmatrix} \gone(v, w) \\ \gtwo(v, w) \end{bmatrix} 
\coloneqq \begin{bmatrix} \fone(v, w) - \Lone \cdot v \\ \ftwo(v, w) - \Ltwo \cdot w \end{bmatrix}.
\end{equation}
The Jacobian of the non-linear remainder, $\g$, has the structure:
\begin{equation}
\label{eqn:block-diagonal-L-component}
\g'(y)= \f'(y)-\Lb =
\begin{bmatrix} \Jone-\Lone & \pfrac{\fone}{w} \\ \pfrac{\ftwo}{v} & \Jtwo - \Ltwo \end{bmatrix}.
\end{equation}

\begin{remark}
Applying $\varphi_k$ to the block-diagonal matrix $\Lb$ from  \eqref{eqn:block-diagonal-L-additive} is equivalent to applying the matrix function to each individual block: 
\begin{equation}
\varphi_k(h \Lb)
= \begin{bmatrix} \varphi_k(h \Lone) & \Zero \\  \Zero & \varphi_k(h \Ltwo) \end{bmatrix}.
\end{equation}
This idea is central to obtaining computational optimizations for partitioned-exponential methods that are unavailable in the case of unpartitioned methods. 
\end{remark}

The exponential Runge-Kutta scheme \eqref{eqn:exp-RK} applied to \eqref{eqn:component-split-ode-semilinear} leads to the following efficient implementation:
\begin{equation}
\label{eqn:exp-RK-component-split}
\begin{split}
\begin{bmatrix} V_i \\ W_i  \end{bmatrix} &= 
\begin{bmatrix} v_{n} + h \, c_i\, \varphi_{1}(c_i h \Lone)\, \fone(v_n,w_n) +  h \, \sum_{j = 2}^{i - 1} a_{i,j}(h \Lone)\, \Big( \gone(V_j,W_j) - \gone(v_n,w_n) \Big) \\ 
w_{n} + h \, c_i\, \varphi_{1}(c_i h \Ltwo)\, \ftwo(v_n,w_n) + h \, \sum_{j = 2}^{i - 1} a_{i,j}(h \Ltwo)\, \Big( \gtwo(V_j,W_j) - \gtwo(v_n,w_n) \Big) 
\end{bmatrix}, \\
\begin{bmatrix} v_{n+1} \\ w_{n+1} \end{bmatrix}  &=\begin{bmatrix} v_{n} + h \,  \varphi_{1}(h \Lone)\, \fone(v_n,w_n) +  h \, \sum_{j = 2}^{i - 1} b_{j}(h \Lone)\, \Big( \gone(V_j,W_j) - \gone(v_n,w_n) \Big) \\ 
w_{n} + h \, \varphi_{1}(h \Ltwo)\, \ftwo(v_n,w_n) + h \, \sum_{j = 2}^{i - 1} b_{j}(h \Ltwo)\, \Big( \gtwo(V_j,W_j) - \gtwo(v_n,w_n) \Big) \end{bmatrix}.
\end{split}
\end{equation} 

\begin{remark}
In \eqref{eqn:exp-RK-component-split} the matrix functions of $\Lone$ are applied to $\fone$, $\gone$, and the matrix functions of $\Ltwo$ are applied to $\ftwo$, $\gtwo$.
\end{remark}

\begin{remark}
The formulation \eqref{eqn:exp-RK-component-split} can be extended to any number of partitions. 
\end{remark}

\subsection{Convergence considerations for component partitioned methods}
\label{sec:convergence-component-partitioned}

Following \cite{luan2014,luan2014a,luan2014b,hochbruck2009} we extend Assumption \ref{ass:one} as follows.

\begin{assumption}
\label{ass:one-partitioned}
There exist constants $M^{\{1\}}$, $M^{\{2\}}$, $\omega^{\{1\}}$, and $\omega^{\{2\}}$, such that
\begin{equation}
\label{eqn:semigroup-partitioned}
\Vert e^{t\,\Lone} \Vert \le M^{\{1\}}\, e^{t\,\omega^{\{1\}}}, \quad 
\Vert e^{t\,\Ltwo} \Vert \le M^{\{2\}}\, e^{t\,\omega^{\{2\}}}, \quad t\ge 0.
\end{equation}
Consequently, Assumption \ref{ass:one} holds for the linear operator $\Lb$ defined in \eqref{eqn:block-diagonal-L-additive}.
\end{assumption}

We also require that Assumption \ref{ass:two} holds. We are particularly interested in the case where $\Lone=\Jone$ and $\Ltwo=\Jtwo$. This implies that the off-diagonal blocks of the Jacobian in equation \eqref{eqn:block-diagonal-L-component} are Lipschitz continuous, with moderate size Lipschitz constants:
\[
\begin{split}
\frac{\partial \gone(v, w)}{\partial w} = \frac{\partial \fone(v, w)}{\partial w}, &
\left\Vert \frac{\partial \fone}{\partial w}(v_a, w_a) - \frac{\partial \fone}{\partial w}(v_b, w_b) \right\Vert \le
\ell_1\,\left( \left\Vert v_a-v_b \right\Vert + \left\Vert w_a-w_b \right\Vert \right), \\
\frac{\partial \gtwo(v, w)}{\partial v} = \frac{\partial \ftwo(v, w)}{\partial v}, &
\left\Vert \frac{\partial \ftwo}{\partial v}(v_a, w_a) - \frac{\partial \ftwo}{\partial c}(v_b, w_b) \right\Vert \le
\ell_2\,\left( \left\Vert v_a-v_b \right\Vert + \left\Vert w_a-w_b \right\Vert \right).
\end{split}
\]


If each component of partitioned system \eqref{eqn:component-split-ode-semilinear} has a dynamics that is stiff, but the coupling terms are non-stiff, then Assumption \ref{ass:two} holds and the convergence of the split scheme \eqref{eqn:exp-RK-component-split} follows from the general theory \cite{luan2014,luan2014a,luan2014b,hochbruck2009}. Increasing the stiffness in the coupling terms leads to a decrease of the time step required for convergence.


\section{Split exponential Runge-Kutta schemes for additively partitioned systems}
\label{sec:partitioned_exprk_methods}

Application of a standard exponential Runge-Kutta method \eqref{eqn:exp-RK} to the additively partitioned system \eqref{eqn:additively-split-ode} requires matrix operations $\varphi_k(h \Lb)$ on a linear operator that approximates the sum of the individual Jacobians \eqref{eqn:jacobians}
\[
\Lb \approx \Jone + \Jtwo.
\]
We seek to develop formulations of the scheme \eqref{eqn:exp-RK} that only perform matrix operations $\varphi_k(h \Lone)$, $\varphi_k(h \Ltwo)$ on individual Jacobian approximations $\Lone \approx \Jone$, $\Ltwo \approx \Jtwo$.
In order to achieve this we first consider a reformulation of exponential Runge-Kutta methods \eqref{eqn:exp-RK}.

\subsection{An alternative formulation of exponential Runge-Kutta methods}

We now derive an alternative formulation of \eqref{eqn:exp-RK} that will prove useful in the implementation of partitioned schemes.

\begin{theorem}
The exponential Runge--Kutta method \eqref{eqn:exp-RK} can be rewritten as:
\begin{equation}
\label{eqn:exp-RK-simple}
\begin{split}
Y_i &= y_n + h\,\boldsymbol{\alpha}_{i,1}(h\Lb)\, f(y_n)  + h\,\sum_{j=2}^{i-1} \boldsymbol{\alpha}_{i,j}(h \Lb)\,\bigg(f(Y_{j}) - f(y_n) \bigg), \quad i = 1,\dots,s, \\
y_{n + 1} &= y_n + h \, \boldsymbol{\beta}_{1}(h \Lb) \, f(y_n)   + h \, \sum_{j=2}^s\boldsymbol{\beta}_{j}(h \Lb) \, \bigg(f(Y_{j}) - f(y_n) \bigg).
\end{split}
\end{equation} 
We denote:
\[
\mathbf{E}_{2:s,2:s}(z) = \bigg( \Id_{s-1} + z\, \mathbf{A}_{2:s,2:s}(z) \bigg)^{-1}.
\]
Since $\mathbf{A}_{2:s,2:s}$ is strictly lower triangular, $\mathbf{E}_{2:s,2:s}$ is lower triangular and involves only multiplications of elements of $\mathbf{A}$. The coefficients of \eqref{eqn:exp-RK-simple} are defined as:
\[
\boldsymbol{\alpha}_{1,1:s}(z) = \Zero, \quad
\boldsymbol{\alpha}_{2:s,1}(z) = \mathbf{E}_{2:s,2:s}(z)\,\begin{bmatrix} c_2\, \varphi_{1}(c_2 z) \\ \vdots \\ c_s\, \varphi_{1}(c_s z) \end{bmatrix}, \quad
\boldsymbol{\alpha}_{2:s,2:s}(z) = \mathbf{E}_{2:s,2:s}(z)\,\mathbf{A}_{2:s,2:s}(z),
\]
\[
\boldsymbol{\beta}_1(z) = \varphi_{1}(z) - z\,\sum_{j=2}^s \mathbf{b}_{j}(z)\, \boldsymbol{\alpha}_{j,1}(z), \quad
\boldsymbol{\beta}_{2:s}^T(z)  = \mathbf{b}_{2:s}^T(z)\,\mathbf{E}_{2:s,2:s}(z).
\]
\label{thm:transformation_theorem}
\end{theorem}

\begin{proof}
Using equations \eqref{eqn:lin-nonlin-split-ode}, the method \eqref{eqn:exp-RK} can be written, after a rearrangement of the internal stages, as:
\begin{equation*}
\begin{split}
Y_i - y_n &=  h \, c_i \varphi_{1}(c_i h \Lb) f(y_n) + h \, \sum_{j = 2}^{i - 1} a_{i,j}(h \Lb)\, \Big( f(Y_j) - f(y_n) \Big)  - \sum_{j = 2}^{i - 1} a_{i,j}(h \Lb)\, h \Lb \, \Big( Y_j - y_n \Big), \\
y_{n + 1} &= y_n + h \,  \varphi_{1}(h \Lb) f(y_n) + h \, \sum_{j = 2}^{s} b_j(h \Lb) \, \Big( f(Y_j) - f(y_n) \Big)  - \sum_{j = 2}^{s} b_j(h \Lb) \, h \Lb\, \Big( Y_j - y_n \Big).
\end{split}
\end{equation*} 
Since $Y_1 = y_n$, we have that $c_1 = 0$, and rest of the internal stages can be stacked and rearranged after multiplication by $\mathbf{E}_{2:s,2:s}$ as follows:
\begin{equation*}
\begin{split}
\begin{bmatrix} Y_2 - y_n \\  \vdots \\ Y_s - y_n  \end{bmatrix} &=  
h\,\mathbf{E}_{2:s,2:s}(h \Lb)\,\begin{bmatrix} c_2\, \varphi_{1}(c_2 h \Lb) \\ \vdots \\ c_s\, \varphi_{1}(c_s h \Lb) \end{bmatrix} f(y_n)  + h\, \mathbf{E}_{2:s,2:s}(h \Lb)\,\mathbf{A}_{2:s,2:s}(h \Lb)\, 
\begin{bmatrix} f(Y_{2}) - f(y_n) \\ \vdots \\ f(Y_{s}) - f(y_n) \end{bmatrix} \\
&=  
h\,\boldsymbol{\alpha}_{2:s,1}(h\Lb)\, f(y_n)  + h\,\boldsymbol{\alpha}_{2:s,2:s}(h \Lb)\,\begin{bmatrix} f(Y_{2}) - f(y_n) \\ \vdots \\ f(Y_{s}) - f(y_n) \end{bmatrix}.
\end{split}
\end{equation*} 
The final one step solution reads:
\begin{equation*}
\begin{split}
y_{n + 1} &= y_n + h \,  \varphi_{1}(h \Lb) f(y_n) + h \, \mathbf{b}_{2:s}^T(h \Lb) \, \begin{bmatrix} f(Y_{2}) - f(y_n) \\ \vdots \\ f(Y_{s}) - f(y_n) \end{bmatrix} 
 - \mathbf{b}_{2:s}^T(h \Lb) \, h \Lb\, \begin{bmatrix} Y_2 - y_n \\  \vdots \\ Y_s - y_n  \end{bmatrix}.
\end{split}
\end{equation*} 
Substituting the expression for $[Y_2 - y_n, \hdots, Y_s - y_n]^{T}$ from above:
\begin{equation*}
\begin{split}
y_{n + 1}  &= y_n + h \,  \varphi_{1}(h \Lb) f(y_n) + h \, \mathbf{b}_{2:s}^T(h \Lb) \, \begin{bmatrix} f(Y_{2}) - f(y_n) \\ \vdots \\ f(Y_{s}) - f(y_n) \end{bmatrix} \\
&\qquad - h\,\mathbf{b}_{2:s}^T(h \Lb)\, \boldsymbol{\alpha}_{2:s}(h\Lb) \, h \Lb\, f(y_n) 
- h\,\mathbf{b}_{2:s}^T(h \Lb) \, h \Lb\, \boldsymbol{\alpha}_{2:s,2:s}(h \Lb)\,\begin{bmatrix} f(Y_{2}) - f(y_n) \\ \vdots \\ f(Y_{s}) - f(y_n) \end{bmatrix} 
\end{split}
\end{equation*} 
After collecting the terms, 
\begin{equation*}
    \begin{split}
        y_{n + 1} &= y_n + h \, \bigg[ \varphi_{1}(h \Lb) - \mathbf{b}_{2:s}^T(h \Lb)\, \boldsymbol{\alpha}_{2:s,1}(h\Lb) \, h \Lb\bigg] \, f(y_n)  \\
&\qquad + h \, \mathbf{b}_{2:s}^T(h \Lb) \bigg[\Id -  h \Lb\, \boldsymbol{\alpha}_{2:s,2:s}(h \Lb) \bigg]\, \begin{bmatrix} f(Y_{2}) - f(y_n) \\ \vdots \\ f(Y_{s}) - f(y_n) \end{bmatrix}.
    \end{split}
\end{equation*} 
It can be shown that:
\[
\Id -  z\, \boldsymbol{\alpha}_{2:s,2:s}(z) = \Id -  z\, \bigg( \Id + z\, \mathbf{A}_{2:s,2:s}(z) \bigg)^{-1} \mathbf{A}_{2:s,2:s}(z) =  \bigg( \Id + z\, \mathbf{A}_{2:s,2:s}(z) \bigg)^{-1}.
\]
Consequently, the final one step solution is written as:
\begin{equation*}
\begin{split}
y_{n + 1} &= y_n + h \, \boldsymbol{\beta}_{1}(h \Lb) \, f(y_n)   + h \, \boldsymbol{\beta}_{2:s}^T(h \Lb) \, \begin{bmatrix} f(Y_{2}) - f(y_n) \\ \vdots \\ f(Y_{s}) - f(y_n) \end{bmatrix}.
\end{split}
\end{equation*} 
\end{proof}

\begin{example}

For example, for $s=2$, we have:
\[
\begin{split}
\boldsymbol{\alpha}(z) &=\begin{bmatrix}
0 & 0  \\
c_2 \varphi_{1}(c_2 z) & 0 
\end{bmatrix}, \quad
\boldsymbol{\beta}(z) =
\begin{bmatrix}
\varphi_{1}(z) - b_2(z)\,  z\, c_2\, \varphi_{1}(c_2 z) \\
b_2(z) 
 \end{bmatrix}.
\end{split}
\]
For $s=3$, we have:
\[
\begin{split}
\boldsymbol{\alpha}(z) &=\begin{bmatrix}
0 & 0 & 0 \\
c_2 \varphi_{1}(c_2 z) & 0 & 0  \\
c_3 \varphi_{1}(c_3 z)-c_2 \,z a_{3,2}(z) \varphi_{1}(c_2 z) & a_{3,2}(z) & 0 
\end{bmatrix}, \\
\boldsymbol{\beta}(z) &=
\begin{bmatrix}
\varphi_{1}(z)- b_3(z)\, \left(z c_3 \varphi_{1}(c_3 z)-c_2 z^2 a_{3,2}(z) \varphi_{1}(c_2 z)\right) 
 - b_2(z)\,  z\, c_2\, \varphi_{1}(c_2 z) \\
 - b_3(z)\, z a_{3,2}(z) +b_2(z) \\
 b_3(z)
 \end{bmatrix}.
\end{split}
\]
\end{example}

\subsection{PEXPRK methods}
%
We next build partitioned-exponential Runge--Kutta methods using the \textit{transformed} exponential Runge--Kutta method  formulation \eqref{eqn:exp-RK-simple}. The \textit{transformed} scheme \eqref{eqn:exp-RK-simple} applied to the additively partitioned system in component form \eqref{eqn:component-form-additive-ode} reads:
\begin{equation*}
\label{eqn:exp-RK-simple-component}
\begin{split}
U_i &\coloneqq V_i + W_i, \qquad
Y_i \coloneqq  \begin{bmatrix}
V_i\\
W_i
\end{bmatrix}, \\
\begin{bmatrix} V_i \\  W_i \end{bmatrix} &= \begin{bmatrix} v_{n} \\  w_{n} \end{bmatrix} 
 +  h \, \begin{bmatrix}  \boldsymbol{\alpha}_{i,1}(h\Lone)\, \fone(y_{n}) \\  \boldsymbol{\alpha}_{i,1}(h\Ltwo)\, \ftwo(y_{n})  \end{bmatrix} 
 + h\,\sum_{j=2}^{i-1} \begin{bmatrix}  \boldsymbol{\alpha}_{i,j}(h \Lone)\,\bigg(\fone(Y_{j}) - \fone(y_{n}) \bigg) \\
 \boldsymbol{\alpha}_{i,j}(h \Ltwo)\,\bigg( \ftwo(Y_{j}) - \ftwo(y_{n}) \bigg)  \end{bmatrix}, \quad i = 1, \hdots, s, \\
\begin{bmatrix} v_{n+1} \\  w_{n+1} \end{bmatrix}  &= \begin{bmatrix} v_{n} \\  w_{n} \end{bmatrix}  
+  h \, \begin{bmatrix}  \boldsymbol{\beta}_{1}(h\Lone)\, \fone(y_{n}) \\  \boldsymbol{\beta}_{1}(h\Ltwo)\, \ftwo(y_{n})  \end{bmatrix} 
 + h\,\sum_{j=2}^{s} \begin{bmatrix}  \boldsymbol{\beta}_{j}(h \Lone)\,\bigg(\fone(Y_{j}) - \fone(y_{n}) \bigg) \\
 \boldsymbol{\beta}_{j}(h \Ltwo)\,\bigg( \ftwo(Y_{j}) - \ftwo(y_{n}) \bigg)  \end{bmatrix}.
\end{split}
\end{equation*} 
Here $V_i$ and $W_i$ are the internal stages of the corresponding component ODEs, $Y_i$ are the internal stages of the method obtained by adding the internal stages and the final one step solutions of the methods applied to the component ODEs, $v_n$ and $w_n$ are the current solutions of the component ODEs, $y_{n}$ is the current solution of the original ODE, $v_{n+1}$ and $w_{n+1}$ are the final one step solutions of the methods applied to the component ODEs.

By adding the internal stages and the final solutions, the above expressions lead to the partitioned Exponential Runge--Kutta formulation.
\begin{definition}
A partitioned Exponential Runge--Kutta (PEXPRK) method reads:
\begin{equation}
\label{eqn:exp-RK-simple-additive}
\begin{split}
U_i &= u_{n} +  h \, \boldsymbol{\alpha}_{i,1}(h\Lone)\, \fone(u_{n}) +  h \, \boldsymbol{\alpha}_{i,1}(h\Ltwo)\, \ftwo(u_{n})  \\
& + h\,\sum_{j=2}^{i-1} \boldsymbol{\alpha}_{i,j}(h \Lone)\,\bigg(\fone(U_{j}) - \fone(u_{n}) \bigg)
+ h\,\sum_{j=2}^{i-1} \boldsymbol{\alpha}_{i,j}(h \Ltwo)\,\bigg( \ftwo(U_{j}) - \ftwo(u_{n}) \bigg), \quad i = 1, \hdots, s, \\
u_{n+1} &= u_{n} 
+  h \, \boldsymbol{\beta}_{1}(h\Lone)\, \fone(u_{n}) +  h \,  \boldsymbol{\beta}_{1}(h\Ltwo)\, \ftwo(u_{n}) \\
& + h\,\sum_{j=2}^{s}  \boldsymbol{\beta}_{j}(h \Lone)\,\bigg(\fone(U_{j}) - \fone(u_{n}) \bigg) 
+ h\,\sum_{j=2}^{s}  \boldsymbol{\beta}_{j}(h \Ltwo)\,\bigg( \ftwo(U_{j}) - \ftwo(u_{n}) \bigg).
\end{split}
\end{equation} 
\end{definition}
Observe that method \eqref{eqn:exp-RK-simple-additive} applies matrix functions of $\Lone$ only to $\fone$, and matrix functions of $\Ltwo$ only to $\ftwo$. In this regard, it is a partitioned method.

\begin{remark}[Systems with multiple partitions]
For a $P$-way partitioned system \eqref{eqn:additively-split-ode}
\begin{equation}
\label{eqn:pway-split-ode}
u' = \f(u) = \sum_{p=1}^P \f^{\{p\}}(u), \qquad u(\tn) = u_n,
\end{equation}
the PEXPRK method \eqref{eqn:exp-RK-simple-additive} becomes:
\begin{equation}
\label{eqn:exp-RK-multiple-additive}
\begin{split}
U_i &= u_{n} +  h \, \sum_{p=1}^P \boldsymbol{\alpha}_{i,1}(h\Lb^{\{p\}})\, \f^{\{p\}}(u_{n})   \\
& + h\,\sum_{p=1}^P \sum_{j=2}^{i-1} \boldsymbol{\alpha}_{i,j}(h \Lb^{\{p\}})\,\bigg(\f^{\{p\}}(U_{j}) - \f^{\{p\}}(u_{n}) \bigg), \quad i = 1, \hdots, s, \\
u_{n+1} &= u_{n} 
+  h \,\sum_{p=1}^P  \boldsymbol{\beta}_{1}(h\Lb^{\{p\}})\, \f^{\{p\}}(u_{n}) + h\,\sum_{p=1}^P \sum_{j=2}^{s}  \boldsymbol{\beta}_{j}(h \Lb^{\{p\}})\,\bigg(\f^{\{p\}}(U_{j}) - \f^{\{p\}}(u_{n}) \bigg).
\end{split}
\end{equation} 
\end{remark}

\begin{theorem}[Partitioned method written in terms of residuals]
\label{thm:pexprk-in-residuals}
The PEXPRK method \eqref{eqn:exp-RK-multiple-additive} can be written in terms of residual differences as follows:
\begin{equation}
\label{eqn:pexpw-in-residuals}
\begin{split}
U_i &= u_{n} +  h \, \sum_{p=1}^P \overline{\boldsymbol{\alpha}}^{\{p\}}_{i,1}(h \Lb^{\{p\}})\, \f^{\{p\}}(u_{n})   \\
& + h\,\sum_{p=1}^P \sum_{j=2}^{i-1} \overline{\boldsymbol{\alpha}}^{\{p\}}_{i,j}(h \Lb^{\{p\}})\,\bigg(\g^{\{p\}}(U_{j}) - \g^{\{p\}}(u_{n}) \bigg), \quad i = 1, \hdots, s, \\
u_{n+1} &= u_{n} 
+  h \,\sum_{p=1}^P  \overline{\boldsymbol{\beta}}^{\{p\}}_{1}(h\Lb^{\{p\}})\, \f^{\{p\}}(u_{n}) + h\,\sum_{p=1}^P \sum_{j=2}^{s}  \overline{\boldsymbol{\beta}}_{j}^{\{p\}}(h \Lb^{\{p\}})\,\bigg(\g^{\{p\}}(U_{j}) - \g^{\{p\}}(u_{n}) \bigg).
\end{split}
\end{equation} 
The coefficients depend on all matrices $\Lb^{\{1\}}$ to $\Lb^{\{P\}}$ and are computed as follows:
\begin{equation}
\label{eqn:pexpw-in-residuals-coeffs}
\begin{split}
\overline{\mathbf{E}}(h\Lb^{\{1\}} \dots h\Lb^{\{P\}}) &= \bigg( \mathbf{I}_{s-1} - \sum_{p=1}^P   h\Lb^{\{p\}}\,\boldsymbol{\alpha}_{2:s,2:s}(h\Lb^{\{p\}}) \bigg)^{-1}, \\
\overline{\boldsymbol{\alpha}}_{2:s,1:s}^{\{p\}}(h\Lb^{\{1\}} \dots h\Lb^{\{P\}}) &\coloneqq \overline{\mathbf{E}}(h\Lb^{\{1\}} \dots h\Lb^{\{P\}}) \,\boldsymbol{\alpha}_{2:s,1:s}(h\Lb^{\{p\}}), \\
\left(\overline{\boldsymbol{\beta}}_{1:s}^{\{p\}}\right)^T(h\Lb^{\{1\}} \dots h\Lb^{\{P\}}) &\coloneqq \boldsymbol{\beta}^T_{1:s}(h\Lb^{\{p\}}) +  \left(  \sum_{q=1}^P  h \Lb^{\{q\}} \,\boldsymbol{\beta}^T_{2:s}(h\Lb^{\{q\}}) \right)\,\overline{\boldsymbol{\alpha}}_{2:s,1:s}^{\{p\}} \\
\end{split}
\end{equation} 

Since $\boldsymbol{\alpha}_{2:s,2:s}$ are strictly lower triangular matrices, a formal inversion of the left hand side matrix $\overline{\mathbf{E}}$ is possible, and the inverse elements are products of the entries of the matrices $\boldsymbol{\alpha}(h\Lb^{\{p\}})$. 
\end{theorem}

\begin{proof}
We seek to rewrite \eqref{eqn:exp-RK-simple-additive} in terms of residuals. For this we first write \eqref{eqn:exp-RK-simple-additive} in matrix form: 
\begin{equation*}
\begin{split}
\begin{bmatrix} U_2 - u_{n} \\  \vdots \\ U_s - u_{n}   \end{bmatrix} &=  
\sum_{p=1}^P  h\,\boldsymbol{\alpha}_{2:s,1}(h\Lb^{\{p\}}) \, \f^{\{p\}}(u_{n})   + \sum_{p=1}^P h\, \boldsymbol{\alpha}_{2:s,2:s}(h\Lb^{\{p\}})\,\begin{bmatrix} \f^{\{p\}}(U_{2}) - \f^{\{p\}}(u_{n}) \\ \vdots \\ \f^{\{p\}}(U_{s}) - \f^{\{p\}}(u_{n}) \end{bmatrix}.
\end{split}
\end{equation*} 
Next, we subtract back the linear terms from both sides:
\begin{equation*}
\begin{split}
& \bigg( \mathbf{I}_{s-1} - \sum_{p=1}^P   h\Lb^{\{p\}}\,\boldsymbol{\alpha}_{2:s,2:s}(h\Lb^{\{p\}}) \bigg) \cdot \begin{bmatrix} U_2 - u_{n} \\  \vdots \\ U_s - u_{n}  \end{bmatrix}
 =  
 \sum_{p=1}^P  h\,\boldsymbol{\alpha}_{2:s,1:s}(h\Lb^{\{p\}}) \,\begin{bmatrix} \f^{\{p\}}(u_{n})  \\ \g^{\{p\}}(U_{2}) - \g^{\{p\}}(u_{n}) \\ \vdots \\ \g^{\{p\}}(U_{s}) - \g^{\{p\}}(u_{n}) \end{bmatrix}.
\end{split}
\end{equation*} 
This, together with the coefficient definition \eqref{eqn:pexpw-in-residuals-coeffs}, gives the stage equations in  \eqref{eqn:pexpw-in-residuals}.

From \eqref{eqn:exp-RK-simple-additive} we have:
\begin{equation*}
\begin{split}
u_{n+1} &= u_{n} 
+  \sum_{p=1}^P  h\,\boldsymbol{\beta}^T(h\Lb^{\{p\}}) \,\begin{bmatrix} \f^{\{p\}}(u_{n})  \\ \g^{\{p\}}(U_{2}) - \g^{\{p\}}(u_{n}) \\ \vdots \\ \g^{\{p\}}(U_{s}) - \g^{\{p\}}(u_{n}) \end{bmatrix} +  \left(  \sum_{p=1}^P  h \Lb^{\{p\}}\,\boldsymbol{\beta}^T_{2:s}(h\Lb^{\{p\}}) \right) \,  \begin{bmatrix}  U_{2} - u_{n} \\ \vdots \\ U_{s} - u_{n} \end{bmatrix}, 
\end{split}
\end{equation*} 
which, together with the coefficient definition \eqref{eqn:pexpw-in-residuals-coeffs}, gives the next solution equation in  \eqref{eqn:pexpw-in-residuals}.
\end{proof}

\begin{example}
For example, for $s=3$, we have:
\[
\begin{split}
\overline{\mathbf{E}} = \begin{bmatrix}
\Id & 0  \\
\sum_{k=1}^P h\Lb^{\{k\}} a_{3,2}(h\Lb^{\{k\}})  & \Id
\end{bmatrix},
\end{split}
\]
\[
\overline{\boldsymbol{\alpha}}^{\{p\}} = \begin{bmatrix}
0 & 0 & 0 \\
c_2 \varphi_{1}(c_2 h\Lb^{\{p\}}) & 0 & 0  \\
c_3 \varphi_{1}(c_3 h\Lb^{\{p\}}) + \left(\sum_{k \ne p} h\Lb^{\{k\}} a_{3,2}(h\Lb^{\{k\}})\right) c_2 \varphi_{1}(c_2 h\Lb^{\{p\}}) & a_{3,2}(h\Lb^{\{p\}}) & 0 
\end{bmatrix},
\]
\[
\overline{\boldsymbol{\beta}}^{\{p\}} = \boldsymbol{\beta}^T_{1:s}(h\Lb^{\{p\}}) +  \left(  \sum_{q=1}^P  h \Lb^{\{q\}} \,\boldsymbol{\beta}^T_{2:s}(h\Lb^{\{q\}}) \right)\,\overline{\boldsymbol{\alpha}}_{2:s,1:s}^{\{p\}},
\]
where 
\[
\boldsymbol{\beta}(z) =
\begin{bmatrix}
	\varphi_{1}(z)- b_3(z)\, \left(z c_3 \varphi_{1}(c_3 z)-c_2 z^2 a_{3,2}(z) \varphi_{1}(c_2 z)\right) 
	- b_2(z)\,  z\, c_2\, \varphi_{1}(c_2 z) \\
	- b_3(z)\, z a_{3,2}(z) +b_2(z) \\
	b_3(z)
\end{bmatrix}.
\]
\end{example}

\subsection{Stability considerations}

The stability analysis of PEXPRK methods is more complex due to the splitting of the Jacobian $\Lone+\Ltwo$ into separate $\Lone$ and $\Ltwo$ parts. We discuss the cause of this complexity, and illustrate with the stability analysis of the PEXPRK method build using exponential Euler as the base scheme.

Consider the additive partitioned system in component form \eqref{eqn:component-form-additive-ode} and  perform  a linear-nonlinear splitting of the form \eqref{eqn:component-split-ode-semilinear}:
\begin{equation}
\label{eqn:additive-split-ode-semilinear}
\begin{split}
y' &= \begin{bmatrix} v' \\ w' \end{bmatrix} =
\begin{bmatrix} \fone(v + w) \\ \ftwo(v + w) \end{bmatrix} =
\begin{bmatrix}  \Lone \cdot v  + \gone(v , w) \\ \Ltwo \cdot w + \gtwo(v , w) \end{bmatrix} 
= \Lb \cdot y  + \g(y),
\end{split}
\end{equation}
where $\Lone$ and $\Ltwo$ are linear operators that contain the stiffnesses of $\fone$ and $\ftwo$, respectively, $\gone$ and $\gtwo$ are the corresponding non-linear remainder functions. 
We have
\begin{equation}
\label{eqn:nonlinear-remainder}
\Lb \coloneqq \begin{bmatrix} \Lone & \Zero \\  \Zero & \Ltwo \end{bmatrix}, \quad
\g(y) \coloneqq \begin{bmatrix} \gone(v, w) \\ \gtwo(v, w) \end{bmatrix} 
\coloneqq \begin{bmatrix} \fone(v+ w) - \Lone \cdot v \\ \ftwo(v+ w) - \Ltwo \cdot w \end{bmatrix}.
\end{equation}
%
%
The standard linear-nonlinear splitting \eqref{eqn:lin-nonlin-split-ode} of \eqref{eqn:additive-split-ode-semilinear} using the full Jacobian evaluated at $y_n$ as the linear operator \eqref{eqn:component-form-additive-ode} leads to 
\begin{equation}
\label{eqn:additive-split-ode-standard-semilinear}
\begin{split}
y' &= \begin{bmatrix} v' \\ w' \end{bmatrix} =
\begin{bmatrix} \fone(v + w) \\ \ftwo(v + w) \end{bmatrix} =
\begin{bmatrix}  \Lone \cdot (v + w)  + \g^{\textnormal{standard}, \{1\}}(v , w) \\ \Ltwo \cdot (v + w) + \g^{\textnormal{standard}, \{2\}}(v , w) \end{bmatrix}. 
\end{split}
\end{equation}
%
%
We make the following additional assumption:
\begin{assumption}
\label{ass:two-for-additive-splitting}
The system \eqref{eqn:additive-split-ode-semilinear} obeys Assumptions \ref{ass:one} and \ref{ass:two}. In particular, there is an initial condition $v_0,w_0$ with $y_0 = [v_0; w_0]$ such that each component equation \eqref{eqn:additive-split-ode-standard-semilinear} admits a sufficiently smooth solution $v,w : [0,T] \to \mathbb{R}^N$. 
\end{assumption}


The following remark shows why the PEXP stability analysis is more challenging. 
\begin{remark}
Higher derivatives of our remainder \eqref{eqn:nonlinear-remainder} and of the standard remainder \eqref{eqn:additive-split-ode-standard-semilinear}  coincide
\[
\g^{(k)}(y) = \left( \g^{\textnormal{standard}} \right)^{(k)} (y), \quad k \ge 2,
\]
and are therefore uniformly bounded if Assumption \ref{ass:two-for-additive-splitting} holds. The first derivatives, however, require a closer look. The Jacobian of the standard remainder \eqref{eqn:additive-split-ode-standard-semilinear} is
\begin{equation}
\label{eqn:jacobian-standard-remainder}
\bigl( \g^{\textnormal{standard}}\bigr)'(y) = 
\begin{bmatrix} {\J}^\one(y) - \Lone & {\J}^\one(y)-\Lone \\ {\J}^\two(y) - \Ltwo & {\J}^\two(y) - \Ltwo \end{bmatrix},
\end{equation}
and is uniformly bounded by Assumption \ref{ass:two-for-additive-splitting}.
The Jacobian of our nonlinear reminder \eqref{eqn:nonlinear-remainder}
\begin{equation}
\g'(y) = 
\begin{bmatrix} {\J}^\one(y)-\Lone & {\J}^\one(y) \\ {\J}^\two(y) & {\J}^\two(y) - \Ltwo \end{bmatrix}
\end{equation}
is not, due to the presence of the stiff Jacobians (and not Jacobian differences) in off-diagonal positions. 
\end{remark}
While the residual Jacobian is not nicely bounded, its product with an analytical matrix function is.
\begin{remark}
The product of a function $h \varphi_k(h\Lb)$ with $k \ge 1$ with the residual Jacobian reads:
\begin{equation}
\label{eqn:phi-jac}
\begin{split}
h \varphi_k(h\Lb)\,\g'(u) &= 
\begin{bmatrix} h \varphi_k(h \Lone) & \Zero \\  \Zero & h \varphi_k(h \Ltwo) \end{bmatrix} \cdot
\begin{bmatrix} {\J}^\one-\Lone & {\J}^\one \\ {\J}^\two & {\J}^\two - \Ltwo \end{bmatrix} \\
&= 
\begin{bmatrix} h \varphi_k(h \Lone)\, \left( {\J}^\one-\Lone \right) & h \varphi_k(h \Lone)\,{\J}^\one \\ h \varphi_k(h \Ltwo)\,{\J}^\two & h \varphi_k(h \Ltwo)\, \left( {\J}^\two - \Ltwo \right) \end{bmatrix}.
\end{split}
\end{equation}
The diagonal blocks are uniformly bounded by Assumption \ref{ass:two-for-additive-splitting}. The off-diagonal blocks are also uniformly bounded, since each term on the right of
\begin{equation}
\label{eqn:phi-jac-bound}
\begin{split}
h \varphi_k(h \Lone)\,{\J}^\one &= h \varphi_k(h \Lone)\, \left( {\J}^\one-\Lone \right) + \varphi_k(h \Lone)\,h \Lone \\
&= h \varphi_k(h \Lone)\, \left( {\J}^\one-\Lone \right) +  \left(\varphi_{k-1}(h \Lone) - \frac{\Id}{(k - 1)!}\right),
\end{split}
\end{equation}
is uniformly bounded.
\end{remark}

Consider the exponential Euler method applied to \eqref{eqn:additive-split-ode-semilinear}--\eqref{eqn:nonlinear-remainder}:
\begin{equation*}
\begin{split}
\begin{bmatrix} v_{n+1} \\ w_{n+1} \end{bmatrix} &= \begin{bmatrix} v_{n} \\ w_{n} \end{bmatrix} +
\begin{bmatrix} h \varphi_1(h \Lone) & \Zero \\  \Zero & h \varphi_1(h \Ltwo) \end{bmatrix} \cdot \begin{bmatrix} \fone(v_{n} + w_{n}) \\ \ftwo(v_{n} + w_{n}) \end{bmatrix} \\
&=
\begin{bmatrix} v_{n} \\ w_{n} \end{bmatrix} +
\begin{bmatrix}  h \varphi_1(h \Lone) \, \left( \Lone (v_{n} + w_{n}) + \g^{\textnormal{standard}\,\{1\}}(v_{n}+w_{n})  \right) \\ h \varphi_1(h \Ltwo)\,\left( \Ltwo (v_{n} + w_{n}) + \g^{\textnormal{standard}\,\{2\}}(v_{n}+w_{n})  \right) \end{bmatrix},
\end{split}
\end{equation*}
and, after adding the component equations together,
\begin{equation}
\label{eqn:exp-Euler}
\begin{split}
u_{n+1}  &= u_{n}  +
  h \varphi_1(h \Lone) \, \left( \Lone \, u_{n} + \g^{\textnormal{standard}\,\{1\}}(u_{n})  \right) 
  + h \varphi_1(h \Ltwo)\,\left( \Ltwo \, u_{n} + \g^{\textnormal{standard}\,\{2\}}(u_{n})  \right) \\
  &= \left( e^{h \Lone} + e^{h \Ltwo} - \Id \right)\,u_{n}  + h \varphi_1(h \Lone) \, \g^{\textnormal{standard}\,\{1\}}(u_{n})
  + h \varphi_1(h \Ltwo)\, \g^{\textnormal{standard}\,\{2\}}(u_{n}).
\end{split}
\end{equation}
Assume, for simplicity, that $h$, $\Lone$, and $\Ltwo$ remain the same throughout the integration, and let 
\begin{equation}
\label{eqn:exp-Euler-M}
\mathbf{M} \coloneqq e^{h \Lone} + e^{h \Ltwo} - \Id.
\end{equation}
The difference of two solutions of the exponential Euler scheme \eqref{eqn:exp-Euler} started with initial conditions $u_0$ and $u_0+\Delta u_0$ is:
\begin{equation*}
\begin{split}
\Delta u_{n+1}  
  &= \mathbf{M}\,\Delta u_{n}  + h \varphi_1(h \Lone) \, \Delta\g^{\textnormal{standard}\,\{1\}}(u_{n})
  + h \varphi_1(h \Ltwo)\, \Delta\g^{\textnormal{standard}\,\{2\}}(u_{n}) \\
  &= \mathbf{M}^n\,\Delta u_{0} + h \,\sum_{k=0}^n \mathbf{M}^k\,\varphi_1(h \Lone) \, \Delta\g^{\textnormal{standard}\,\{1\}}(u_{n-k}) \\
  &\quad + h \,\sum_{k=0}^n \mathbf{M}^k\,\varphi_1(h \Ltwo) \, \Delta\g^{\textnormal{standard}\,\{2\}}(u_{n-k}) \\
\Vert \Delta u_{n+1} \Vert &\le \Vert  \mathbf{M}^n\Vert \, \Vert \Delta u_{0} \Vert 
+ h \,\left( \Vert \varphi_1(h \Lone)\Vert  \, \ell^{\{1\}} + \Vert \varphi_1(h \Ltwo)\Vert  \, \ell^{\{1\}}\right)\,\sum_{k=0}^n \Vert \mathbf{M}^k\Vert \,\Vert \Delta u_{n-k}\Vert,  
\end{split}
\end{equation*}
where $\ell^{\{i\}}$ is the (small) Lipschitz constant of the smooth function $\g^{\textnormal{standard}\,\{i\}}$ for $i=1,2$. The exponential Euler scheme \eqref{eqn:exp-Euler} is stable if the matrix $\mathbf{M}$ defined in \eqref{eqn:exp-Euler-M} is power bounded.

%
%
\section{Particular methods}
\label{sec:particular_methods}
We now turn our attention to building particular methods of partitioned exponential Runge--Kutta type. We start with known exponential Runge--Kutta methods \eqref{eqn:exp-RK} that satisfy the stiff order conditions given in Table \ref{tbl:expRK-order-conditions} (for up to order four). Next we use Theorem \ref{thm:transformation_theorem} to obtain coefficients of the \textit{transformed} exponential Runge--Kutta method \eqref{eqn:exp-RK-simple}. Finally, we apply the \textit{transformed} method to construct PEXPRK schemes \eqref{eqn:exp-RK-simple-additive}. Figure \ref{fig:pexprk_construction_process} illustrates this procedure.

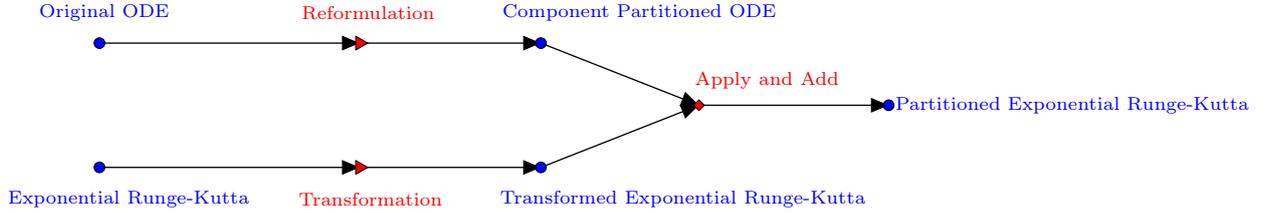
\begin{figure}[htb!]
\definecolor{ffqqqq}{rgb}{1,0,0}
\definecolor{rvwvcq}{rgb}{0,0,1}
\begin{tikzpicture}[line cap=round,line join=round,>=triangle 45,x=1cm,y=1cm,scale=0.8]
\draw [->,line width=0.5pt] (-12.8,-0.07) -- (-8.52,-0.07);
\draw [->,line width=0.5pt] (-8.52,-0.07) -- (-5.54,-0.07);
\draw [->,line width=0.5pt] (-12.8,1.99) -- (-8.52,1.99);
\draw [->,line width=0.5pt] (-8.52,1.99) -- (-5.54,1.99);
\draw [->,line width=0.5pt] (-5.54,1.99) -- (-2.94,0.96);
\draw [->,line width=0.5pt] (-5.54,-0.07) -- (-2.94,0.96);
\draw [->,line width=0.5pt] (-2.94,0.96) -- (0.18,0.96);
\begin{scriptsize}
\draw [fill=rvwvcq] (-12.8,1.99) circle (2.5pt);
\draw[color=rvwvcq] (-12.72,2.5) node {Original ODE};
\draw [fill=rvwvcq] (-12.8,-0.07) circle (2.5pt);
\draw[color=rvwvcq] (-12.32,-0.6) node {Exponential Runge-Kutta};
\draw [fill=ffqqqq,shift={(-8.52,-0.07)},rotate=270] (0,0) ++(0 pt,3.75pt) -- ++(3.2475952641916446pt,-5.625pt)--++(-6.495190528383289pt,0 pt) -- ++(3.2475952641916446pt,5.625pt);
\draw[color=ffqqqq] (-8.34,-0.6) node {Transformation};
\draw [fill=rvwvcq] (-5.54,-0.07) circle (2.5pt);
\draw[color=rvwvcq] (-3.2,-0.6) node {Transformed Exponential Runge-Kutta};
\draw [fill=ffqqqq,shift={(-8.52,1.99)},rotate=270] (0,0) ++(0 pt,3.75pt) -- ++(3.2475952641916446pt,-5.625pt)--++(-6.495190528383289pt,0 pt) -- ++(3.2475952641916446pt,5.625pt);
\draw[color=ffqqqq] (-8.38,2.5) node {Reformulation};
\draw [fill=rvwvcq] (-5.54,1.99) circle (2.5pt);
\draw[color=rvwvcq] (-3.92,2.5) node {Component Partitioned ODE};
\draw [fill=ffqqqq] (-2.94,0.96) ++(-2.5pt,0 pt) -- ++(2.5pt,2.5pt)--++(2.5pt,-2.5pt)--++(-2.5pt,-2.5pt)--++(-2.5pt,2.5pt);
\draw[color=ffqqqq] (-1.82,1.38) node {Apply and Add};
\draw [fill=rvwvcq] (0.18,0.96) circle (2.5pt);
\draw[color=rvwvcq] (3.2,0.96) node {Partitioned Exponential Runge-Kutta};
\end{scriptsize}
\end{tikzpicture}
\caption{Procedure to build a partitioned exponential Runge--Kutta method}
\label{fig:pexprk_construction_process}
\end{figure}

We build partitioned exponential Runge--Kutta schemes of orders two, three and four. For the second order scheme, we go all the way and build a partitioned method. For the third and fourth order, we only provide the coefficients of the original and the \textit{transformed} exponential Runge--Kutta method. It is then straightforward to build a partitioned method with the coefficients of the \textit{transformed} method.

\subsection{Second order partitioned exponential Runge--Kutta scheme}
We start with the second order exponential Runge-Kutta scheme \eqref{eqn:exp-RK} given by the coefficients \cite[Section 5.1]{Hochbruck_2005_expRK}:
\[
\boldsymbol{c} = \begin{bmatrix} 0 \\ 1 \end{bmatrix}, \quad
\boldsymbol{a}(z) = \begin{bmatrix} 0 & 0 \\ \varphi_1(z) & 0 \end{bmatrix}, \quad
\boldsymbol{b}(z) = \begin{bmatrix} \varphi_1(z)-\varphi_2(z) \\ \varphi_2(z) \end{bmatrix},
\]
that satisfy order conditions 1., 2a., and 2b. in Table \ref{tbl:expRK-order-conditions}. Here, $s = 2$, $\mathbf{A}_{2:s,2:s} = 0$ and $\mathbf{E}_{2:s,2:s} = 1$. The coefficients of the transformed method are as follows:
\[
    \begin{split}
\boldsymbol{\alpha}(z) = \begin{bmatrix} 0 & 0 \\
\varphi_1(z) & 0 \end{bmatrix}, \qquad
\boldsymbol{\beta}(z) = \begin{bmatrix} \varphi_1(z) - z \varphi_2(z)\varphi_1(z) \\
\varphi_2(z) \end{bmatrix}.
    \end{split}
\]
The partitioned implementation \eqref{eqn:exp-RK-simple-additive} of the \textit{transformed} method reads:
\begin{equation}
\label{eqn:partitioned-exp2}
\begin{split}
U_1 &= u_n, \\
U_2 
&=  u_n + h \,\sum_{p=1}^2 \varphi_{1}(h \Lp)\, \fp(u_n),\\
u_{n+1}   
 &= u_{n} 
+  h \,\sum_{p=1}^2 \left( 2\varphi_{1}(h \Lp) - \varphi^2_{1}(h \Lp)\right)\, \fp(u_n)   \\
& \quad + h \,\sum_{p=1}^2  \varphi_2(h \Lp) \, \Big( \fp(U_2) - \fp(u_n) \Big).
\end{split}
\end{equation}
As in Theorem \ref{thm:pexprk-in-residuals}, the method \eqref{eqn:partitioned-exp2} can be written in terms of residuals as follows:
\begin{equation}
\label{eqn:partitioned-exp2-solution-residuals}
\begin{split}
u_{n+1}   
&= u_{n} +  h \,\sum_{p=1}^2 \left( \sum_{p'\ne p} \varphi_{1}(h \Lb^{\{p'\}})\right)\, \varphi_{1}(h \Lp)\, \fp(u_n)   \\
& \quad + h \,\sum_{p=1}^2  \varphi_2(h \Lp) \, \Big( \gp(U_2) - \gp(u_n) \Big).
\end{split}
\end{equation}

\subsection{Third order partitioned exponential Runge--Kutta scheme}

Consider the third order method by Hochbruck and Ostermann \cite{Hochbruck_2005_expRK} given by the Butcher tableau:

\[
\renewcommand\arraystretch{1.5}
\begin{array}{l|cccc}
    0 & 0 & 0 & 0 & \\[8pt]
    c_2 & c_2\varphi_{1,2} & 0 & 0 &\\[8pt]
    \tfrac{2}{3} & \tfrac{2}{3}\varphi_{1,3} - \tfrac{4}{9c_2}\varphi_{2,3}  & \tfrac{4}{9c_2}\varphi_{2,3} & 0 & \\[8pt]
    \hline
    & \varphi_{1} - \tfrac{3}{2} \varphi_{2} & 0 & \tfrac{3}{2} \varphi_{2} & \\
\end{array}.
\]

Here, $\varphi_{j} \equiv \varphi_j(z)$ and $\varphi_{j, k} \equiv \varphi_j(c_k  z)$. For $c_2 = 2/3$, the \textit{transformed} method coefficients, upon applying Theorem \ref{thm:transformation_theorem}, are given shown in Method \ref{method:3rdorder_transformed_unpartitioned}.

\begin{method}
    \begin{equation*}
    \begin{split}        
    \boldsymbol{\alpha} = \begin{bmatrix}
    0 & 0 & 0 \\
                      \frac{2 \varphi_{1,2}}{3} &   0 & 0 \\
                      \frac{2}{9} \left(3 \varphi_{1,3} - 2z \varphi_{1,2} \varphi_{2,3}\right) &  \frac{2 \varphi_{2,3}}{3} & 0 \\
            \end{bmatrix}, \\[8pt]
    \boldsymbol{\beta} = \begin{bmatrix}
    \frac{1}{3} z \varphi_2 \left(2 z \varphi_{1,2} \varphi_{2,3}-3 \varphi_{1,3}\right)+\varphi_1(z) \\
        -z \varphi_2 \varphi_{2,3}    \\
        \frac{3 \varphi_2}{2} 
        \end{bmatrix}.
    \end{split}
\end{equation*}
\caption{Third order transformed, unpartitioned method. \label{method:3rdorder_transformed_unpartitioned}}
\end{method}

\subsection{Fourth order partitioned exponential Runge--Kutta scheme}

We consider next the fourth order method by Hochbruck and Ostermann \cite{Hochbruck_2005_expRK} given by the Butcher tableau:
\[
\renewcommand\arraystretch{1.5}
\begin{array}{l|ccccc}
    0 & 0 & 0 & 0 & 0 & 0\\[8pt]
    \tfrac{1}{2} & \tfrac{1}{2}\varphi_{1,2} & 0 & 0 & 0 & 0\\[8pt]
    \tfrac{1}{2} & \tfrac{1}{2}\varphi_{1,3} - \varphi_{2,3}  & \varphi_{2,3} & 0 & 0 & 0\\[8pt]
    1 & \varphi_{1,4} - 2 \varphi_{2,4} & \varphi_{2,4} & \varphi_{2,4} & 0 & 0\\[8pt]
    \tfrac{1}{2} & \tfrac{1}{2}\varphi_{1,5} - 2 a_{5,2} - a_{5,4} & a_{5,2} & a_{5,2} & \tfrac{1}{4} \varphi_{2,5} - a_{5,2} & 0\\[8pt]\hline
        & \varphi_{1} - 3 \varphi_{2} + 4 \varphi_{3} & 0 & 0 & -\varphi_{2} + 4 \varphi_{3} & 4 \varphi_{2} - 8 \varphi_{3}\\
\end{array},
\]
Applying Theorem \ref{thm:transformation_theorem} the transformed method has the coefficients given in Method \ref{method:4thorder_transformed_unpartitioned}.

\begin{method}
    \footnotesize
    \begin{equation*}
    \begin{split}        
        \alpha_{2:s,1} &= \begin{bmatrix}
        \frac{\varphi_{1,2}}{2}\\
        \frac{1}{2} (\varphi_{1,3}-z \varphi_{1,2} \varphi_{2,3})\\
        \frac{1}{2} z   \varphi_{2,4} (\varphi_{1,2} (z \varphi_{2,3}-1)-\varphi_{1,3})+\varphi_{1,4}\\
        \frac{1}{8}
        \bigg(2 z \varphi_{1,4} (\varphi_{2,4}+\varphi_{2,5}-2 (2 \varphi_{3,4}+\varphi_{3,5}))\\
        +z
        \varphi_{1,3} (-\varphi_{2,4} (z (\varphi_{2,4}+\varphi_{2,5}-2 (2 \varphi_{3,4}+\varphi
        _{3,5}))+1)-2 \varphi_{2,5}+4 \varphi_{3,4}+2 \varphi_{3,5})\\
        +z \varphi_{1,2} (z
        \varphi_{2,3}-1) (\varphi_{2,4} (z (\varphi_{2,4}+\varphi_{2,5}-2 (2 \varphi_{3,4}+\varphi
        _{3,5}))+1)+2 (\varphi_{2,5}-2 \varphi_{3,4}-\varphi_{3,5}))
        +4 \varphi
        _{1,5}\bigg)
    \end{bmatrix},\\[8pt]
    \alpha_{2:s,2:s} &= \begin{bmatrix}
             0 & 0 & 0 & 0 \\[8pt]
             \varphi_{2,3} & 0 & 0 & 0 \\[8pt]
             (1-z \varphi_{2,3}) \varphi_{2,4} & \varphi_{2,4} & 0 & 0 \\[8pt]
             -\frac{1}{4} (z \varphi_{2,3}-1) \bigg(2 (\varphi_{2,5}-2 \varphi_{3,4}-\varphi_{3,5}) 
              & \frac{1}{4} \bigg(z (\varphi_{2,4}+\varphi_{2,5}-2 (2 \varphi
               _{3,4}+\varphi_{3,5})) 
               & \frac{1}{4} \bigg(-\varphi_{2,4}-\varphi_{2,5}
                & 0 \\
               + \varphi_{2,4} (z (\varphi_{2,4}+\varphi_{2,5}-2 (2 \varphi_{3,4}+\varphi_{3,5}))+1)\bigg) & \varphi_{2,4}+\varphi_{2,4}+2 \varphi_{2,5}-4 \varphi_{3,4}-2 \varphi_{3,5}\bigg)  & +4 \varphi_{3,4}+2 \varphi_{3,5}\bigg) &
            \end{bmatrix},\\[8pt]
    \beta_{1} &= \varphi_1(z)-\frac{1}{2} z \biggl((\varphi_2-2 \varphi_3) \bigl(2 z \varphi_{1,4} (\varphi_{2,4}+\varphi
    _{2,5}-2 (2 \varphi_{3,4}+\varphi_{3,5}))\\
&\quad    +z \varphi_{1,3} (-\varphi_{2,4} (z
    (\varphi_{2,4}+\varphi_{2,5}-2 (2 \varphi_{3,4}+\varphi_{3,5}))+1)-2 \varphi_{2,5}+4
    \varphi_{3,4}+2 \varphi_{3,5})\\
&\quad    +z \varphi_{1,2} (z \varphi_{2,3}-1) (\varphi_{2,4} (z
    (\varphi_{2,4}+\varphi_{2,5}-2 (2 \varphi_{3,4}+\varphi_{3,5}))+1)+2 (\varphi
    _{2,5}-2 \varphi_{3,4}-\varphi_{3,5}))+4 \varphi_{1,5}\bigr)\\
&\quad    -(\varphi_2-4 \varphi_3)
    (z \varphi_{2,4} (\varphi_{1,2} (z \varphi_{2,3}-1)-\varphi_{1,3})+2 \varphi
    _{1,4})\biggr),\\[8pt]     
   \beta_{2} &=  \frac{1}{4} z \left(z \varphi_{2,3}-1\right) \Big(4 \left(\varphi_2-2 \varphi_3\right) \left(\varphi_{2,4}
    \left(z \left(\varphi_{2,4}+\varphi_{2,5}-2 \left(2 \varphi_{3,4}+\varphi_{3,5}\right)\right)+1\right)+2
    \left(\varphi_{2,5}-2 \varphi_{3,4}-\varphi_{3,5}\right)\right) \\
    &\quad -4 \left(\varphi_2-4 \varphi_3\right) \varphi
    _{2,4}\Big),\\[8pt]
    \beta_{3:s} &= \begin{bmatrix}    z \left(\varphi_2-4 \varphi_3\right) \varphi_{2,4}+z \left(\varphi_2-2 \varphi_3\right) \left(-\varphi
    _{2,4} \left(z \left(\varphi_{2,4}+\varphi_{2,5}-2 \left(2 \varphi_{3,4}+\varphi_{3,5}\right)\right)+1\right)-2
    \varphi_{2,5}+4 \varphi_{3,4}+2 \varphi_{3,5}\right)\\
    z \left(\varphi_2-2 \varphi_3\right) \left(\varphi_{2,4}+\varphi
    _{2,5}-2 \left(2 \varphi_{3,4}+\varphi_{3,5}\right)\right)-\varphi_2+4 \varphi_3\\
    4 \left(\varphi_2-2 \varphi
    _3\right) \end{bmatrix}.
    \end{split}
\end{equation*}
\caption{Fourth order transformed, unpartitioned method. \label{method:4thorder_transformed_unpartitioned}}
\end{method}

%
\section{Computational considerations}
\label{sec:computational-optimizations}

The most expensive computation, outside function and Jacobian evaluation, in one step of an exponential integrator is the evaluation of the $\varphi_k$ function \eqref{eqn:phi_function_definition} times vector products, which involves evaluating the matrix exponential operator. There is a significant body of work in computing matrix exponentials as evidenced in the review article by Moler et al. \cite{moler2003}. A number of authors have made additional contributions to the topic  \cite{niesen2012,al-mohy2011,bader2015,higham2005,higham2009,schmelzer2007,skaflestad2009,almohy2010}. 

In our implementation of exponential Runge--Kutta methods, we evaluate the $\varphi_k(h \gamma J) v$ products using Krylov-subspace based approximations. If $N$ is the number of variables in the state-space, we use the modified Arnoldi algorithm \cite{vorst2003} to build an $M$-dimensional Krylov-subspace ($M \ll N$):  
\[
    \mathcal{K}_{M} = \{v, \J v, \J^2 v, \J^3 v, \hdots, \J^{M - 1} v\},
\]
returning an ${N \times M}$ orthonormal matrix $V$ that spans $\mathcal{K}_{M}$ and an ${M \times M}$ upper-Hessenberg matrix, $H_M$. Using these, the $\varphi_k$ product is approximated according to \cite{Sidje_1998,Sandu_2014_expK,Sandu_2019_EPIRKW,saad1992,hochbruck1997}: 
\[
    \varphi_k(h\gamma J) v \approx \norm{v} V_M \varphi_{k}(h \gamma H_M) e_1. 
\] 
We implement an adaptive Arnoldi algorithm \cite{saad1992} that evaluates the error at $m$-th iteration, $s_m = \varphi_1(h \gamma J) v - \norm{v} V_m \varphi_{k}(h \gamma H_m) e_1$, and stopping when the error under a chosen tolerance.  Even though we only evaluate the error for $\varphi_1$, we argue that this should be sufficient for higher order $\varphi_k$ functions as they must converge faster according to their series expansion. Furthermore, since the computation of error to limit the size of the subspace (and the number of iterations of Arnoldi) in itself relies on performing a matrix exponential on the upper-Hessenberg matrix $H_m$, we only compute the error at some pre-determined indices \cite[Section 6.4]{Hochbruck_1998_exp}, where the cost of computing the error roughly doubles the total cost of all preceding error computations.

In the remainder of this section, we attempt to establish guidelines for when partitioned exponential methods will be computationally more efficient per timestep that traditional exponential methods applied to the full system.

\textit{\textbf{Cost of unpartitioned method.}}
The cost of an Arnoldi process, excluding the cost of error estimation code, to construct $V_M$ and $H_M$ is $\landauO\left(MN^2 + M^2N + MN\right)$, assuming that the matrices are dense and a matrix vector product in the full space requires $\landauO(N^2)$ operations.. Since we only compute the error at indices where the cost of computing the error roughly doubles the cost of all previous error computations, the total cost of computing the error after $M$ iterations is $\landauO\left(M^3\right)$. This cost comes from the fact that the error estimation routine requires  the exponentiation of the $H_M$ matrix, with a constant number of rows and columns augmented around it which we ignore in its cost. 
    
The total cost per timestep of an unpartitioned exponential method that evaluates $u$ $\varphi$ matrix function vector products is $\landauO\left(u \times (M N^2 + M^2 N + M N + M^3)\right)$ (ignoring a constant number of vector adds). Since $N \gg M$ the cost is approximately $\sim u\,M\, N^2$.

\textit{\textbf{Cost of component partitioned method.}} Consider a $P$-way component partitioned system with components of size $N_i$ , $i=1,\dots,P$, respectively, such that $N = N_1 + \dots + N_P$. Consider a partitioned method that performs $p$ matrix-vector operations on each partition.  Assume next that the stiff components are divided proportionately amongst the partitions, and the Arnoldi processes need the same number of $M$ iterations to converge. The cost per step of the partitioned method is $\sim p\,M\, N_1^2 + \dots + p\,M\, N_P^2$. The partitioned method is more effective than the full method when 
\[
p < u \frac{\left(\sum_{i=1}^P N_i\right)^2}{\sum_{i=1}^P N_i^2}.
\]
When $N_i = N/P$ the partitioned method is more effective when $p < u/P^2$.

\textit{\textbf{Cost of additively partitioned method.}} The analytical functions are applied to matrices $h\Lp$ of the same size. Computational savings are realized due to the simpler structure of individual matrices. First, the number of Krylov iterations for computing each individual $\varphi_k(h\Lp)\,v$ can be much smaller than that of the full, more complex Jacobian $\varphi_k(h\Lb)\,v$. Next, the sparsity structure of individual matrices can be much more favorable than the sparsity structure of the sum. Individual component matrices can have special structure (such as block diagonal) that  can further speed up the calculations.

Further computational optimizations specific to the application of partitioned exponential methods to reaction-diffusion systems are discussed in greater detail in our earlier work \cite{narayanamurthi2019}.
%

\section{Results}
\label{sec:results}

We perform fixed timestep numerical experiments using each exponential Runge--Kutta method described in Section \ref{sec:particular_methods} in the original formulation, transformed formulation, and partitioned formulation (for a test-specific partitioning). In the results that follow, we use the following naming convention for the methods. 

Labels starting with prefix ``exprks'' stand for unpartitioned, exponential Runge--Kutta scheme.  Likewise, labels starting with ``pexprks'' stand for partitioned, exponential Runge--Kutta scheme. We also use the keywords ``orig'' and ``tran'' to indicate whether the unpartitioned methods are using the original coefficients or the {transformed} coefficients. Following this, we specify the design order of the method with the phrase ``$\text{order}\, = \, \text{design order}$'' directly in the label itself. For example, ``$\text{order}\, = \, 2$'' in the label corresponds to the second-order method from Section \ref{sec:particular_methods}. Lastly, we specify the structure of the Jacobian used with methods implemented using the transformed coefficients, i.e., whether it's a full or a block Jacobian. We use the full Jacobian with the original methods but do not specify it in the label. To summarize, the nomenclature for labeling each method in the plots is as follows:
\[
    {\big[\text{exprks or pexprks}\big]} \_
    {\big[\text{orig or tran}\big]}   \_
    {\big[\text{order = 2 or 3 or 4}\big]} \_
    {\big[\text{(Full or Block Jacobian)}\big]}. 
\]

\subsection{Gray-Scott equations} 

The Gray-Scott equations \cite{gray1984}:
\begin{equation}
	\begin{split}
	\pfrac{a}{t} &= d_a \nabla^2 a  -  a b^2 + f(1 - a)  ,\\
	\pfrac{b}{t} &= d_b \nabla^2 b  +  a b^2 - (f + k)b ,\\
	\end{split}
	\label{eqn:gray-scott}
\end{equation}
model the diffusion and reaction of two species $a$ and $b$ over a spatial domain where the chemical reactions are:
\begin{equation*}
	\begin{split}
	A + 2B &\rightarrow C  ,\qquad
	B \rightarrow C.
    \end{split}
\end{equation*}  
In equation \eqref{eqn:gray-scott} $a$, $b$ are the concentrations of the two chemical species, $d_a$ and $d_b$ are the diffusion coefficients of the species, $f$ is the feed rate of $a$, and $f + k$ is drain rate of $b$. Here we use the parameter values $f = 0.04$, $k = 0.06$, $d_a = 2$, $d_b = 1$.

The spatial domain is $[0, 1] \times [0, 1]$ and the timespan for experiments is $[0, 0.262144]$. Initial conditions for the model are given by the expressions $v(x,y,0) = 0.4 + 0.1 (x+y) + 0.1 (\sin(10*x) \sin(20*y))$ and $w(x,y,0) = 0.4 + 0.1 (x+y) + 0.1 (\cos(10*x) \cos(20*y))$. The boundary conditions are periodic.

For our experiments, we use a 2D finite difference discretized Gray-Scott model with a grid resolution of $300 \times 300$. Each method was tested for the set of stepsizes given by the MATLAB expression \texttt{0.262144 .* 2.^{(-1:-1:-6)}},  which roughly corresponds to the decimal range \num{1e-1} to \num{4e-3}. The reference solution was computed using ODE15s with absolute and relative tolerance set to \num{1.0e-14}. Four different partitions of the discretized model were tested. For all experiments we used the additive implementation of the partitioned methods.

\subsection{Component partitioning by chemical species\label{sec:component_partition_by_species}}

We first consider a component partitioning of \eqref{eqn:gray-scott}, where each component is associated with a different chemical species. In equation \eqref{eqn:component-form-additive-ode} $v$ and $w$ variables corresponding to this splitting are
\begin{equation}
\label{eqn:GS-partitioning-species}
v = \begin{bmatrix}
\bar{a}\\
\mathbf{0}
\end{bmatrix}, \quad w = \begin{bmatrix}
\mathbf{0} \\
\bar{b}
\end{bmatrix},
\end{equation}
where $\bar{a}$ and $\bar{b}$ are the state variables corresponding to the chemical species $A$ and $B$ after spatial discretization. With $\mathds{L}$ representing the discrete Laplacian operator, and $\mathds{1}$, a column of ones of appropriate dimension, the right-hand side functions $\fone$ and $\ftwo$ can be written as 
\[
\fone(v + w) = \begin{bmatrix}
d_a * \mathds{L} * \bar{a} - \bar{a} \odot \bar{b} \odot \bar{b} + f * (\mathds{1} - \bar{a}) \\
\mathbf{0}
\end{bmatrix},\quad
\ftwo(v + w) = \begin{bmatrix}
\mathbf{0}\\
d_b * \mathds{L} * \bar{b} + \bar{a} \odot \bar{b} \odot \bar{b} - (f + k) * \bar{b}
\end{bmatrix}.
\]
The corresponding linear operators in \eqref{eqn:exp-RK-simple-additive} are  
\[
\Lone = \begin{bmatrix}
d_a * \mathds{L} - \textnormal{diag}(\bar{b} \odot \bar{b}) - f * \mathds{I} & \mathbf{0}\\
\mathbf{0} & \mathbf{0}
\end{bmatrix}, \quad
\Ltwo = \begin{bmatrix}
\mathbf{0} & \mathbf{0} \\
\mathbf{0} & d_b * \mathds{L} + 2 * \textnormal{diag}(\bar{a} \odot \bar{b}) - (f + k) * \mathds{I}
\end{bmatrix},
\]
where $\mathds{I}$ is the identity matrix and $\textnormal{diag}$ is a function that returns a diagonal matrix with arguments placed along the main diagonal. 

Figure \ref{fig:gray-scott-results-var-partitioned} shows the results of the fixed timestep experiments. All methods show the correct order of convergence. As a validation of the \textit{transformation} Theorem \ref{thm:transformation_theorem}, we observe that the original and transformed methods converge at the same rate as expected.

\begin{figure}[htb!]
	\centering
    \includegraphics[scale=0.4]{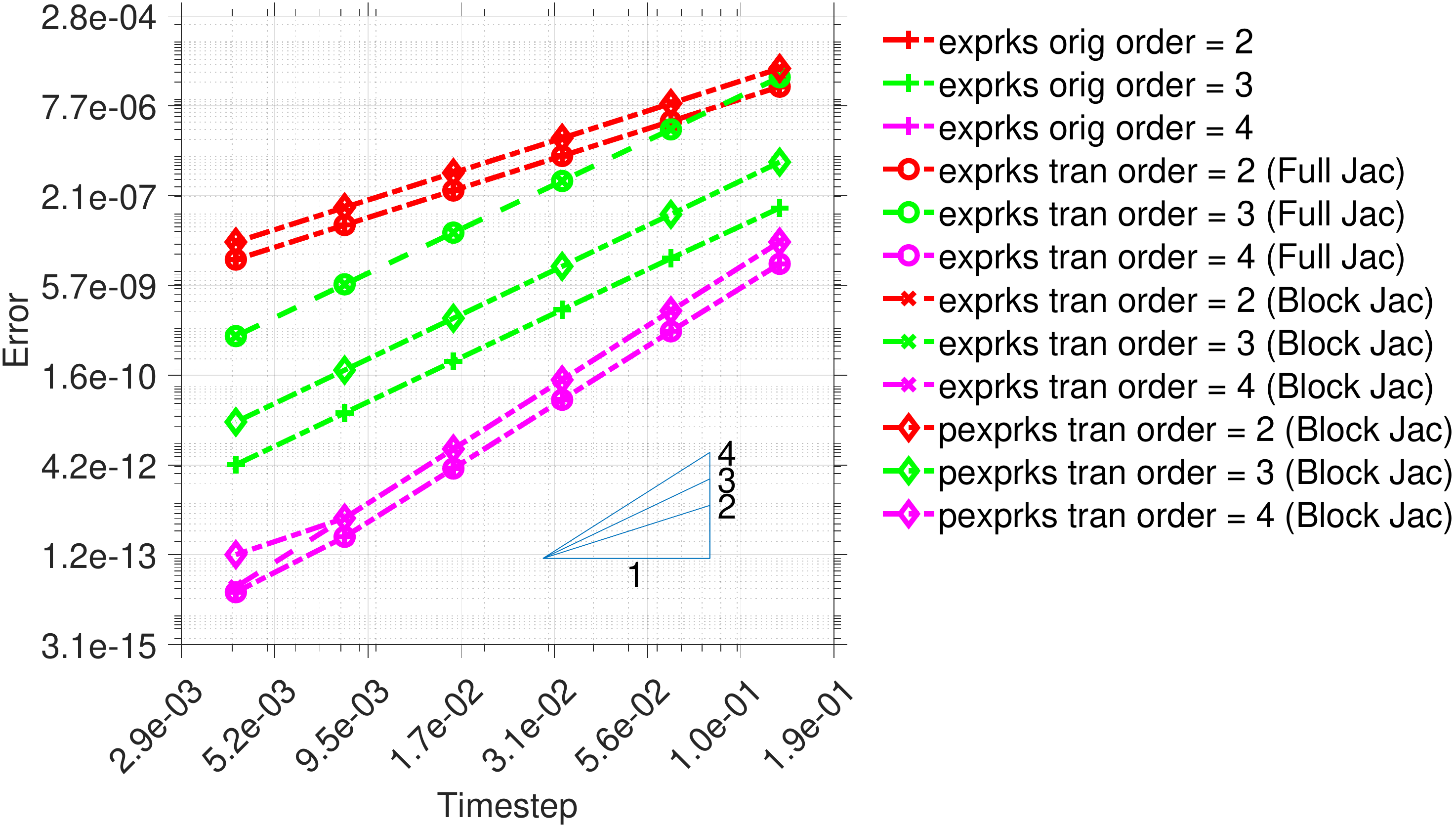}
    \caption{Gray-Scott equations \eqref{eqn:gray-scott}  partitioned by chemical species \eqref{eqn:GS-partitioning-species}.}
    \label{fig:gray-scott-results-var-partitioned}
\end{figure}

\begin{remark}
	For component partitioned systems, an unpartitioned exponential Runge--Kutta scheme using a block-diagonal Jacobian, $\Lb$, implemented as a \textit{transformed} method is equivalent to its partitioned counterpart. If we set 
		\begin{equation}
		\label{eqn:GS-full-state}
		\begin{split}
		& \bar{Y} =\begin{bmatrix} \bar{a} \\ \bar{b} \end{bmatrix}, \quad 
		f(\bar{Y}) = \begin{bmatrix}
		d_a * \mathds{L} * \bar{a} - \bar{a} \odot \bar{b} \odot \bar{b} + f * (\mathds{1} - \bar{a})\\
		d_b * \mathds{L} * \bar{b} + \bar{a} \odot \bar{b} \odot \bar{b} - (f + k) * \bar{b}
		\end{bmatrix}, \\
		& \Lb = \begin{bmatrix}
		d_a * \mathds{L} - \textnormal{diag}(\bar{b} \odot \bar{b}) - f * \mathds{I} & \mathbf{0}\\
		\mathbf{0} & d_b * \mathds{L} + 2 * \textnormal{diag}(\bar{a} \odot \bar{b}) - (f + k) * \mathds{I}
		\end{bmatrix},
		\end{split}
		\end{equation}
		the results are identical to those obtained by the partitioned methods as shown in Figure \ref{fig:gray-scott-results-var-partitioned}. 
\end{remark}

\subsection{Component partitioning by splitting the spatial domain\label{sec:component_partition_by_space}}

In this experiment, we partition the variables by their spatial location into two domains as shown in Figure \ref{fig:grayscott_spatial_partition}. ODE functions and linear operators are updated accordingly. In order to achieve this, we make use of permutation matrices which reorder the state variables first by grid location and then by species, such that all variables in subdomain $\Omega_1$ are the first $N/2$ variables in the permuted state vector, and the variables in subdomain $\Omega_2$ are the last $N/2$ ones.
%
%
\begin{figure}[htb!]
	\centering
	\definecolor{ududff}{rgb}{0.30196078431372547,0.30196078431372547,1.}
\definecolor{zzttqq}{rgb}{0.6,0.2,0.}
\definecolor{cqcqcq}{rgb}{0.7529411764705882,0.7529411764705882,0.7529411764705882}
\begin{tikzpicture}[line cap=round,line join=round,>=triangle 45,x=1.0cm,y=1.0cm,scale=0.5]
\fill[line width=2.pt,color=zzttqq,fill=zzttqq,fill opacity=0.10000000149011612] (-2.,-2.) -- (6.,-2.) -- (6.,6.) -- (-2.,6.) -- cycle;
\draw [line width=2.pt,color=zzttqq] (-2.,-2.)-- (6.,-2.);
\draw [line width=2.pt,color=zzttqq] (6.,-2.)-- (6.,6.);
\draw [line width=2.pt,color=zzttqq] (6.,6.)-- (-2.,6.);
\draw [line width=2.pt,color=zzttqq] (-2.,6.)-- (-2.,-2.);
\draw [line width=2.pt] (-2.,2.)-- (6.,2.);
\draw [->,line width=2.pt] (-4.,-4.) -- (8.,-4.);
\draw [->,line width=2.pt] (-4.,-4.) -- (-4.,8.);
\begin{scriptsize}
\draw [fill=black] (-2.,-2.) circle (2.5pt);
\draw[color=black] (-2.23792,-2.5) node {(0,0)};
\draw [fill=black] (6.,-2.) circle (2.5pt);
\draw[color=black] (6.41358,-2.5) node {(1, 0)};
\draw [fill=black] (6.,6.) circle (2.5pt);
\draw[color=black] (6.33372,6.52261) node {(1,1)};
\draw [fill=black] (-2.,6.) circle (2.5pt);
\draw[color=black] (-2.18468,6.49599) node {(0,1)};
\draw [fill=black] (-2.,2.) circle (2.5pt);
\draw[color=black] (-2.8,2.13031) node {(0, $\frac{1}{2}$)};
\draw [fill=black] (6.,2.) circle (2.5pt);
\draw[color=black] (6.94598,2.10369) node {(1, $\frac{1}{2}$)};
\draw [fill=ududff] (2.,0.) circle (2.5pt);
\draw[color=ududff] (2.64685,0.54642) node {$\Omega_1$};
\draw [fill=ududff] (2.,4.) circle (2.5pt);
\draw[color=ududff] (2.64685,4.56604) node {$\Omega_2$};
\draw [fill=ududff] (-4.,-4.) circle (0.5pt);
\draw [fill=ududff] (8.,-4.) circle (0.5pt);
\draw[color=ududff] (8.19712,-3.72609) node {x};
\draw [fill=ududff] (-4.,8.) circle (0.5pt);
\draw[color=ududff] (-3.8085,8.27953) node {y};
\end{scriptsize}
\end{tikzpicture}
	\caption{Spatial domains when variables are partitioned by location in the numerical grid.}
	\label{fig:grayscott_spatial_partition}
\end{figure}
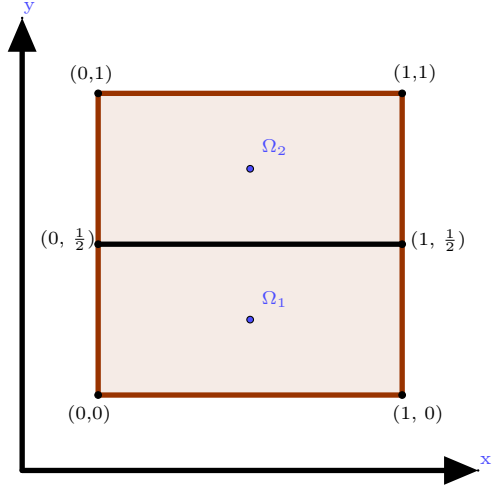

If $\mathbf{P}$ is the permutation matrix which performs the rearrangement of the state variables, the  variables corresponding to this splitting scheme are $v=(\mathbf{P}^\top \bar{Y}) _{1 : N/2}$, $w=(\mathbf{P}^\top \bar{Y}) _{N/2 + 1 : N}$,
%
%
and the corresponding partition is applied to the right hand side function \eqref{eqn:GS-full-state}.
Lastly, if $\mathbf{J}$ is the Jacobian of the right hand side, then the linear operators are defined as
$\Lone=(\mathbf{P}^{\top} \mathbf{J} \mathbf{P})_{1:N/2, 1:N/2}$ and $\Ltwo=(\mathbf{P}^{\top} \mathbf{J} \mathbf{P})_{N/2+1:N, N/2+1:N}$.
%

Figure \ref{fig:gray-scott-results-space-partitioned} shows the results of fixed timestep experiment for the scenario where the discretized model is partitioned by spatial domain. We note that all the methods show the theoretical orders of convergence.
\begin{figure}[htb!]
	\centering
    \includegraphics[scale=0.4]{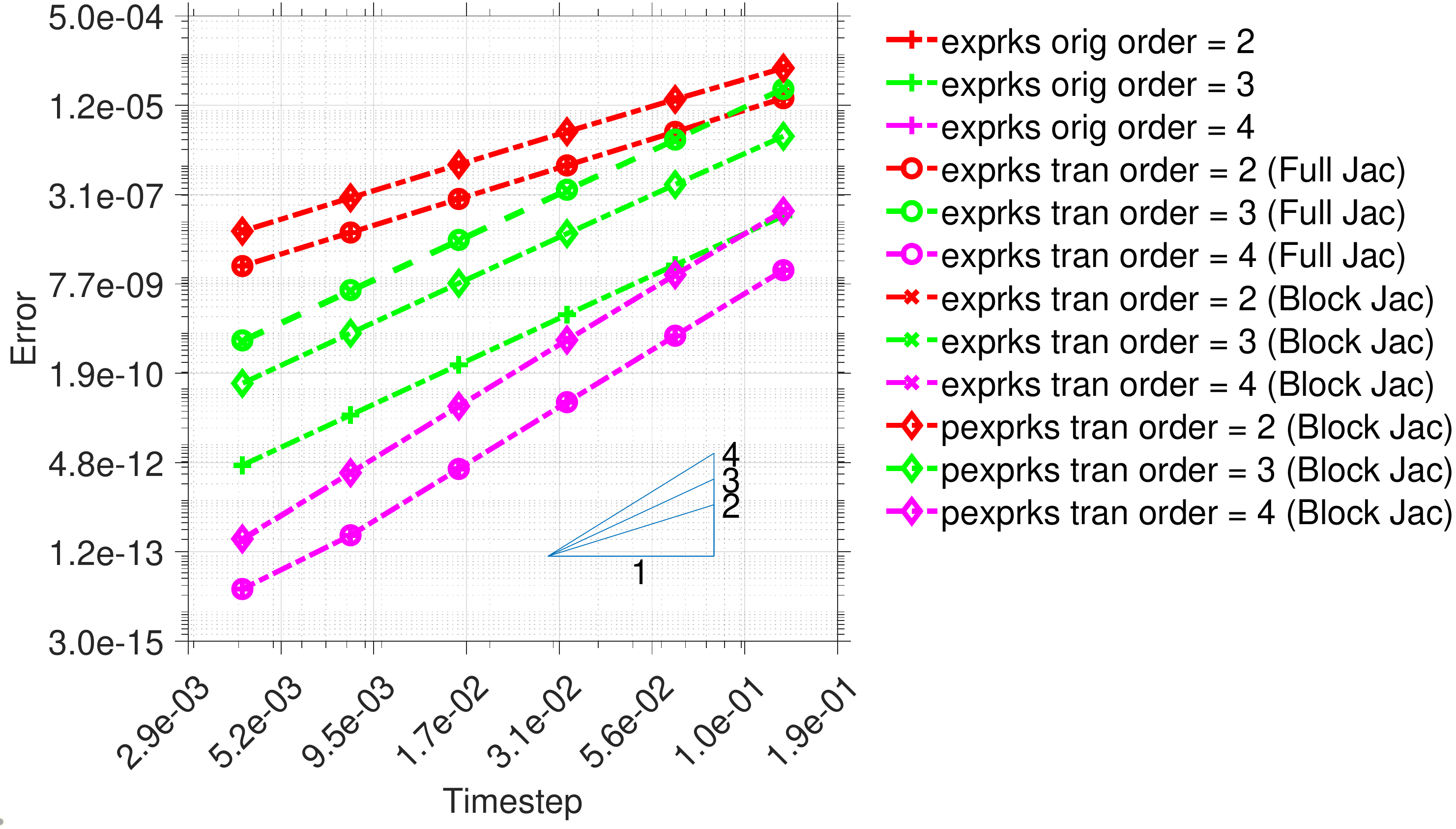}
    \caption{Gray-Scott equations \eqref{eqn:gray-scott}  partitioned by spatial domains on a 300x300 grid, as in Figure \ref{fig:grayscott_spatial_partition}.}
    \label{fig:gray-scott-results-space-partitioned}
\end{figure}

\subsection{Additive partitioning by physical process}

Next, we consider an additive partitioning of the Gray-Scott system. The right-hand side of the PDE \eqref{eqn:gray-scott} models diffusion of two chemical species, $A$ and $B$, in a medium and the interaction between species via reaction terms. For this experiment, we split the system along these two distinct physical processes:
\begin{subequations}
\label{eqn:GS-process-split}
\begin{equation}
\fone(\bar{a},\bar{b}) = \f^{\textnormal{diffusion}} = \begin{bmatrix}
d_a * \mathds{L} * \bar{a} \\
d_b * \mathds{L} * \bar{b}
\end{bmatrix},
\end{equation}
and,
\begin{equation}
\ftwo(\bar{a},\bar{b}) = \f^{\textnormal{reaction}} = \begin{bmatrix}
- \bar{a} \odot \bar{b} \odot \bar{b} + f * (\mathds{1} - \bar{a})\\
+ \bar{a} \odot \bar{b} \odot \bar{b} - (f + k) * \bar{b}
\end{bmatrix}.
\end{equation}
The linear operators corresponding to these right-hand sides are:
\begin{equation}
\Lone = \Lb^{\textnormal{diffusion}} = \begin{bmatrix}
d_a * \mathds{L} & \mathbf{0}\\
\mathbf{0} & d_b * \mathds{L}
\end{bmatrix},
\end{equation}
and
\begin{equation}
\Ltwo = \Lb^{\textnormal{reaction}} =\begin{bmatrix}
- \textnormal{diag}(\bar{b} \odot \bar{b}) - f * \mathds{I} & -2 * \textnormal{diag}(\bar{a} \odot \bar{b})\\
\textnormal{diag}(\bar{b} \odot \bar{b}) & 2 * \textnormal{diag}(\bar{a} \odot \bar{b}) - (f + k) * \mathds{I}
\end{bmatrix}.
\end{equation}
\end{subequations}
Linear algebra takes advantage of the highly regular block structure of $\Lone$ and $\Ltwo$.

Figure \ref{fig:gray-scott-results-phy-partitioned} shows the results of fixed timestep experiment for the above scenario where the discretized model is partitioned by physics, viz., reaction and diffusion. Again, we get full convergence in all the partitioned methods.

\begin{figure}[htb!]
	\centering
    \includegraphics[scale=0.4]{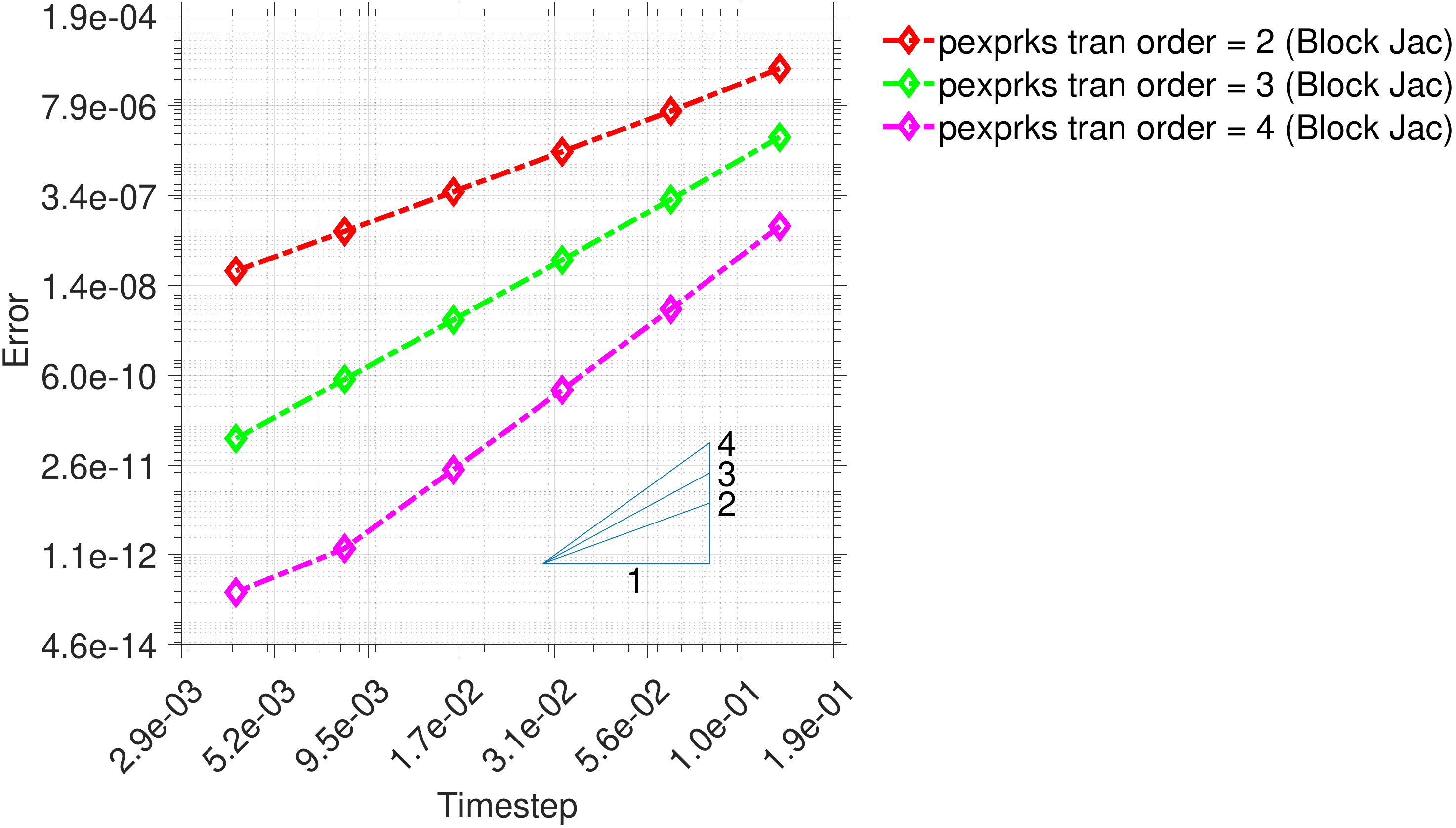}
    \caption{Gray-Scott equations \eqref{eqn:gray-scott} partitioned by physics \eqref{eqn:GS-process-split}.}
    \label{fig:gray-scott-results-phy-partitioned}
\end{figure}

\subsection{Partitioning by physics with reaction treated explicitly, and diffusion treated exponentially}

We note that, for our settings, the reaction terms in the Gray-Scott model \eqref{eqn:gray-scott} are less stiff compared to the diffusion terms. Consequently, we test our partitioned exponential methods with an ``IMEX-like'' configuration \cite{Ascher_1995_IMEX_LMM,Ascher_1997_IMEX_RK}. Specifically, we treat the stiffer diffusion partition exponentially, and the less stiff reaction partition explicitly. We make one minor modification to the previous experiment to implement this. While we define $\fone$, $\ftwo$ and $\Lone$ similarly to the previous experiment, we set $\Ltwo = \mathbf{0}$.

When we set the linear operator to a zero matrix, an Exponential Runge--Kutta method degenerates to the underlying Runge--Kutta method \cite{Hochbruck_2005_expRK}. Likewise, by setting the linear operator of any one partition of a PEXPRK to a zero matrix, we obtain an explicit method in that partition.

Figure \ref{fig:gray-scott-results-phy-explicit-explonential-partitioned} shows the results of fixed timestep experiment for the scenario where the discretized model is partitioned by physics, i.e., into reaction and diffusion, with the reaction partition treated explicitly and the diffusion partition exponentially. We observe the theoretical convergence behavior in all the partitioned methods. 

\begin{figure}[htb!]
	\centering
    \includegraphics[scale=0.4]{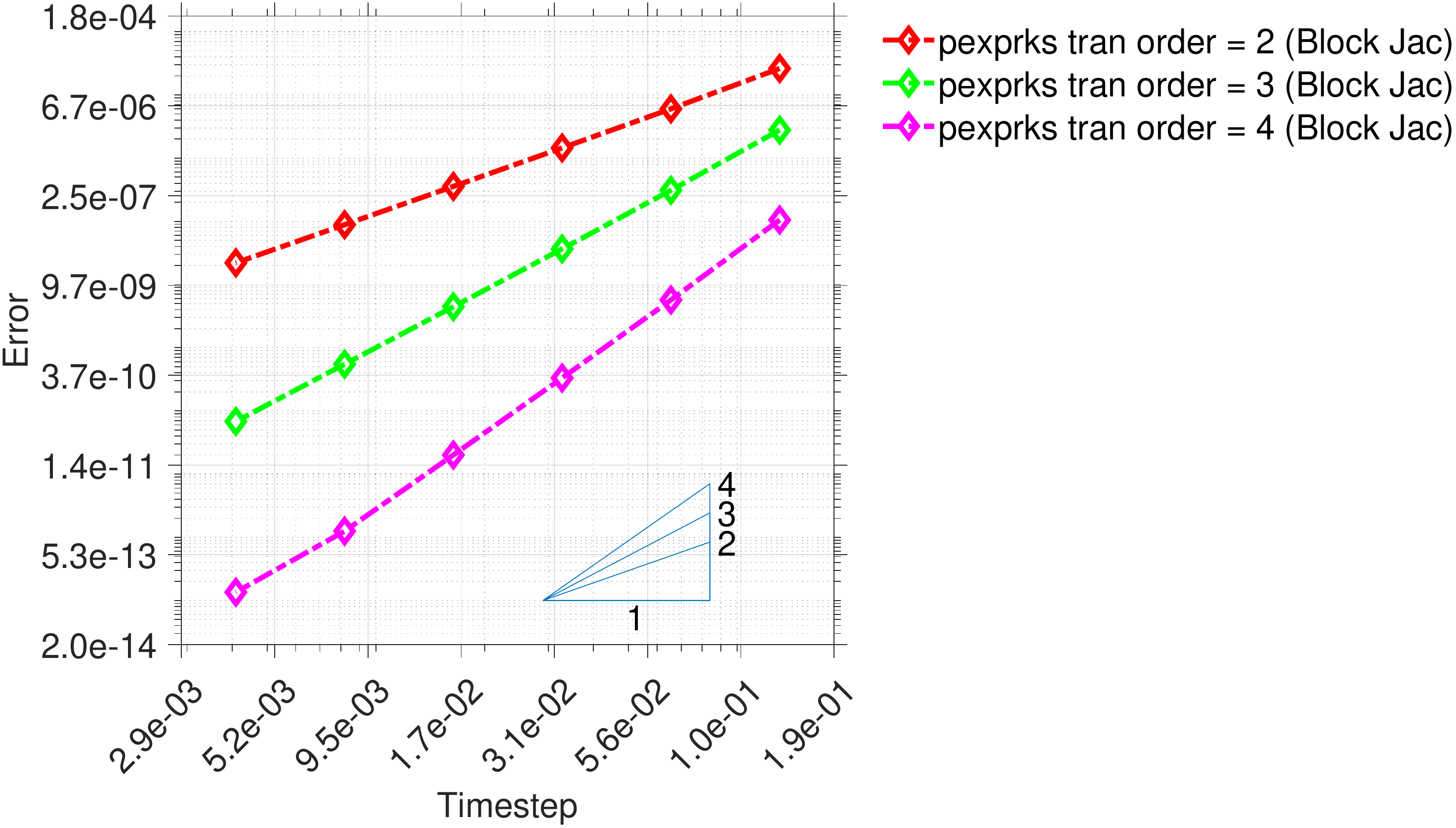}
    \caption{Gray-Scott equations \eqref{eqn:gray-scott} partitioned by physics with the reaction term treated explicitly, and the diffusion term treated exponentially.}
    \label{fig:gray-scott-results-phy-explicit-explonential-partitioned}
\end{figure}
%

\section{Conclusions}
\label{sec:conclusions}

In our earlier work \cite{narayanamurthi2019} we developed classical order conditions theory for partitioned methods that solve each component of a multiphysics problem with a different exponential method. We found that the partitioned methods were amenable to certain computational optimizations that unpartitioned methods could not take advantage of, and that the partitioned methods can be more stable than unpartitioned methods for some parameter settings. However, those methods are challenged by stiff systems with a tight coupling between components. 

In this work we develop an implementation of exponential Runge--Kutta methods that is specific to partitioned systems. The resulting partitioned exponential methods inherit the W stiff order conditions from the base methods \cite{luan2014,luan2014a,luan2014b,hochbruck2009}. A direct application of stiff exponential Runge--Kutta methods to partitioned systems is not immediate due to the properties of the matrix exponential operator and the non-linear terms. Our approach is to introduce a transformation for exponential Runge--Kutta methods that enables the derivation of partitioned methods of that type. We apply the modified methods to component and to additively partitioned systems, using a special approximation of the Jacobian as the linear operator. This approximation allows to evaluate matrix functions on Jacobians of individual  component functions rather than on the combined Jacobian. These matrix functions are combined to obtain the numerical solution. The resulting partitioned exponential methods are of exponential-exponential and exponential-explicit type.

We derive three methods of second, third and fourth order of partitioned exponential type and numerically demonstrate that they converge as expected for different types of partitions of the test problem.  

\section*{Acknowledgements}

This work has been supported in part by NSF through awards NSF ACI–1709727 and NSF CCF–1613905, AFOSR through the award AFOSR DDDAS 15RT1037, and by the Computational Science Laboratory at Virginia Tech.

\section*{References}

\end{document}